\documentclass[a4paper,10pt,twoside]{scrartcl}

\usepackage{amsmath}
\usepackage{amsthm}
\usepackage{amsxtra}
\usepackage{amssymb}

\usepackage[english]{babel}

\usepackage[format=plain, labelfont=bf]{caption}
\usepackage{csquotes}

\usepackage{dsfont}
\usepackage{dblfloatfix}

\usepackage{exscale}

\usepackage{float}

\usepackage{graphicx}
\graphicspath{ {./figures/} }
\usepackage{geometry}
\usepackage{hyperref}

\usepackage[utf8]{inputenc}

\usepackage{latexsym}

\usepackage{mathrsfs}
\usepackage{mathtools}
\usepackage{microtype}

\usepackage[automark]{scrlayer-scrpage}
\usepackage{setspace}
\usepackage{subfigure}

\usepackage{tabularx}
\usepackage{todonotes}

\usepackage{verbatim}

\hypersetup{
	colorlinks=false,
}

\DeclareFontFamily{U}{mathx}{\hyphenchar\font45}
\DeclareFontShape{U}{mathx}{m}{n}{
	<5> <6> <7> <8> <9> <10>
	<10.95> <12> <14.4> <17.28> <20.74> <24.88>
	mathx10
}{}
\DeclareSymbolFont{mathx}{U}{mathx}{m}{n}
\DeclareFontSubstitution{U}{mathx}{m}{n}
\DeclareMathAccent{\widecheck}{0}{mathx}{"71}
\DeclareMathAccent{\wideparen}{0}{mathx}{"75}

\providecommand{\keywords}[1]
{
	\small	
	\textbf{Keywords ---} #1
}

\newcommand{\R}{\mathds{R}}

\newcommand{\Pw}{\mathds{P}}

\newcommand{\Z}{\mathds{Z}}
\newcommand{\E}{\mathds{E}}

\newcommand{\N}{\mathds{N}}

\newcommand{\cB}{\mathcal{B}}
\newcommand{\cC}{\mathcal{C}}

\newcommand{\cE}{\mathcal{E}}
\newcommand{\cI}{\mathcal{I}}

\newcommand{\cF}{\mathcal{F}}

\newcommand{\cM}{\mathcal{M}}

\newcommand{\cP}{\mathcal{P}}

\newcommand{\cT}{\mathcal{T}}

\newcommand{\bfB}{\mathbf{B}}
\newcommand{\bfC}{\mathbf{C}}

\newcommand{\bfX}{\mathbf{X}}

\newcommand{\1}{\mathds{1}}

\theoremstyle{definition}
\newtheorem{definition}{Definition}[section]

\newtheorem{remark}[definition]{Remark}
\newtheorem{openproblem}{Open problem}
\newtheorem{conjecture}[openproblem]{Conjecture}
\theoremstyle{plain}
\newtheorem{theorem}[definition]{Theorem}
\newtheorem{lemma}[definition]{Lemma}
\newtheorem{corollary}[definition]{Corollary}

\newtheorem{proposition}[definition]{Proposition}

\makeatletter
\renewenvironment{proof}[1][\proofname]{%
	\par\pushQED{\qed}\normalfont%
	\topsep6\p@\@plus6\p@\relax
	\trivlist\item[\hskip\labelsep\bfseries#1\@addpunct{.}]%
	\ignorespaces
}{%
	\popQED\endtrivlist\@endpefalse
}
\makeatother

\begin{document}
	\title{Contact process on a dynamical long range percolation}
	\author{M. Seiler\footnote{Frankfurt Institute for Advanced Studies, Ruth-Moufang-Straße 1, 60438 Frankfurt am Main, Germany	E-mail: seiler@fias.uni-frankfurt.de}  \and  A. Sturm\footnote{Institute for Mathematical Stochastics, Georg-August-Universit\"at G\"ottingen, Goldschmidtstr. 7, 37077
			G\"ottingen, Germany E-mail: anja.sturm@mathematik.uni-goettingen.de}}
	\date{\today}
	\maketitle
	\begin{abstract}
			\small{In this paper we introduce a contact process on a dynamical long range percolation (CPDLP) defined on a complete graph $(V,\cE)$. A dynamical long range percolation is a Feller process defined on the edge set $\cE$, which assigns to each edge the state of being open or closed independently. The state of an edge $e$ is updated at rate $v_e$ and is open after the update with probability $p_e$ and closed otherwise. The contact process is then defined on top of this evolving random environment using only open edges for infection while recovery is independent of the background. First, we conclude that an upper invariant law exists and that the phase transitions of survival and non-triviality of the upper invariant coincide. We then formulate a comparison with a contact process with a specific infection kernel which acts as a lower bound. Thus, we obtain an upper bound for the critical infection rate. We also show that if the probability that an edge is open is low for all edges then the CPDLP enters an immunization phase, i.e.\ it will not survive regardless of the value of the infection rate. Furthermore, we show that on $V=\Z$ and under suitable conditions on the rates of the dynamical long range percolation the CPDLP will almost surely die out if the update speed converges to zero for any given infection rate $\lambda$.} 
	\end{abstract}
	\section{Introduction}
The classical contact process on a fixed graph  describes the spread of an infection over time and space. It has been studied intensively and many variations have been considered, see Section \ref{sec:relatedlit} for more background.
In this article we study a \emph{contact processes on a dynamical long range percolation} (CPDLP), in which infections over any distance are possible depending on whether the corresponding edge is present, which also changes dynamically. 

We assume that the underlying graph $G=(V,E)$ is a connected and transitive graph with bounded degree. The graph distance of $G$ is denoted by $d(\cdot,\cdot)$ and  the set of all possible edges by $\cE:=\{e=\{x,y\}: x,y\in V, x\neq y \}$. From now on we consider the complete graph $(V,\cE)$ which we also equip with the original graph distance $d(\cdot,\cdot)$.

There are several notions of transitivity for graphs in the literature. Thus, we specify the notion briefly. Here a graph $G$ is called transitive if for any pair of vertices $x_1,x_2\in V$ and respectively for any pair of edges $e_1,e_2\in E$, there exists a graph automorphism $\phi$ which maps $x_1$ to $x_2$ and $e_1$ to $e_2$.
A graph automorphism is a permutation on $V$ which preserves the graph structure, i.e.\ $\{x,y\} \in E$ iff $\{\phi(x),\phi(y)\} \in E$. In the literature this is sometimes called vertex and edge transitivity.

The CPDLP $(\bfC,\bfB)=(\bfC_t,\bfB_t)_{t\geq 0}$ is a Markov process on $\cP(V)\times \cP(\cE)$, where $\cP(V)$ and $\cP(\cE)$ denote the power sets of $V$ and $\cE$. We equip the space $\cP(V)\times \cP(\cE)$ with the topology induced by pointwise convergence, i.e. $(C_n, B_n) \in \cP(V)\times \cP(\cE)$ converges to $(C,B)$ if $\1_{\{(x,e) \in (C_n, B_n)  \}} \rightarrow \1_{\{(x,e) \in (C,B)  \}}$ for all $(x,e) \in \cP(V)\times \cP(\cE).$  Note that $\cP(V)\times \cP(\cE)$ is a partially ordered space with respect to $``\subset"$. Furthermore, we denote by $``\Rightarrow"$ weak convergence of probability measures on $\cP(V)\times \cP(\cE)$. As usual we denote by $|A|$ the cardinality of a set $A$.

We call $\bfC$ the \emph{infection process}, which takes values in $\cP(V)$. If $x\in\bfC_{t}$ then we call $x$ \textit{infected} at time $t$. The process $\bfB$ describes an evolving edge random environment and takes values in $\cP(\cE)$. Thus, we call $\bfB$ the \emph{background process}. If $e\in \bfB_t$ we call $e$ \emph{open} at time $t$ and \emph{closed} otherwise. Furthermore, we assume that $\bfB$ evolves autonomously of $\bfC$. Given that $\bfB$ is currently in state $B$ the transitions of the infection process $\bfC$ currently in state $C$ are for all $x\in V$, 
\begin{equation}\label{InfectionRatesWithBackground}
	\begin{aligned}
		C&\to C\cup \{x\}	\quad \text{ at rate } \lambda \cdot|\{y\in C:\{x,y\}\in 
		B\}| \text{ and }\\
		C&\to C\backslash \{x\}	\quad\; \; \text{ at rate } r,
	\end{aligned}
\end{equation}
where $\lambda>0$ denotes the \emph{infection rate} and $r>0$ the \emph{recovery rate}. 
We write $\bfC=\bfC^C$ when $\bfC_0=C.$
For the background dynamics  we consider a \emph{dynamical long range percolation}. Let $(\hat{p}_e)_{e\in \cE}\subset[0,1]$ and $(\hat{v}_e)_{e\in \cE}\subset(0,\infty)$ be sequences of real numbers such that $\hat{p}_{\{x,y\}}=\hat{p}_{\{x',y'\}}$ and $\hat{v}_{\{x,y\}}=\hat{v}_{\{x',y'\}}$ if $d(x,y)=d(x',y')$. We exclude the trivial case that $\hat{p}_e=0$ for all $e\in \cE.$
Now the dynamical long range percolation $\bfB$ currently in state $B$ has transitions 
\begin{equation}\label{DymLongRangePer}
	\begin{aligned}
		B&\to B\cup\{e\} \quad \text{ at rate } \hat{v}_{e}\hat{p}_{e}\text{ and }\\
		B&\to B\backslash\{e\}	\quad \text{ at rate } \hat{v}_{e}(1-\hat{p}_{e})
	\end{aligned}
\end{equation}
for all $e\in \cE$. As initial distribution we choose $\bfB_0\sim \pi$, where $\pi$ is the invariant law of $\bfB$ which means that the events $(\{e\in \bfB_t\})_{e\in \cE}$ are independent and $\Pw(e\in \bfB_t)=\hat{p}_{e}$ for all $e\in \cE$ and $t\geq 0$.

We will in particular be interested in the behavior of our process when we scale the percolation and speed kernels. For this we will assume that they are of the form 
\begin{equation}
	\label{scaledkernels}
	\hat{p}_{e}=\hat{p}_{e}(q):=q p_{e}\quad \text{ and } \quad \hat{v}_{e}=\hat{v}_{e}(\gamma):=\gamma v_{e}
\end{equation}
for all $e\in \cE$ for some $\gamma>0$ and $q\in(0,1]$ and fixed kernels $(p_e)_{e\in \cE}\subset[0,1]$ and $(v_e)_{e\in \cE}\subset(0,\infty)$. A long range percolation model assigns to every edge $e\in\cE$ independently the state of being open with probability $\hat{p}_e$ and otherwise closed. The term "dynamical" means that we update the state of every edge $e$ as time evolves. This is done independently for every edge $e=\{x,y\}$ at update speed $\hat{v}_e$ which depends only on the length $d(x,y)$ of that edge.
This yields a translation invariant background dynamic, where we use that the graph $G$ is transitive.

Since we are in a long range setting we need some assumptions regarding the flip rates of the background process to ensure that the CPDLP is well-defined. 

In order to ensure that the transition rates of the infection process are not infinite we need  that at any given time $t$ the neighborhood of any vertex $x$ remains finite. Therefore, we assume that the sequence $(\hat{p}_{e})_{e\in \cE}$ and $(\hat{v}_{e})_{e\in \cE}$ satisfy 
\begin{equation}\label{ExistenceAssumption}
	\sum_{y\in V\backslash\{x\}}\hat{v}_{\{x,y\}}\hat{p}_{\{x,y\}}<\infty \,\,\, \text{ and }\,\,
	\sum_{y\in V\backslash\{x\}}\hat{v}_{\{x,y\}}^{-1}<\infty\,\,\, \text{ for all } x\in V.
\end{equation}
Note that if the kernels are of the form (\ref{scaledkernels}) then $(p_e)_{e\in \cE}\subset[0,1]$ and $(v_e)_{e\in \cE}\subset(0,\infty)$ satisfy these assumptions iff $(\hat{p}_e)_{e\in \cE}\subset[0,1]$ and $(\hat{v}_e)_{e\in \cE}\subset(0,\infty)$ do.
\begin{remark}\label{Rem:pIsSummable}
	The assumptions in (\ref{ExistenceAssumption})  imply that $\hat{v}_{\{x,y\}}\hat{p}_{\{x,y\}}\to 0$ and $\hat{v}_{\{x,y\}}\to \infty$ as $d(x,y)\to \infty$. Since we also have that $\hat{v}_e>0$ for all $e\in \cE$ it follows that $C:=\inf_{y:y\neq x}\hat{v}_{\{x,y\}}>0$, where $C$ does not depend on the choice of $x\in V$ due to translation invariance. This implies together with \eqref{ExistenceAssumption} that
	\begin{equation*}
		0<C\sum_{y\in V\backslash\{x\}}\hat{p}_{\{x,y\}}\leq \sum_{y\in V\backslash\{x\}}\hat{v}_{\{x,y\}}\hat{p}_{\{x,y\}}<\infty,
	\end{equation*}
	i.e.\ that the sequence $(\hat{p}_{\{x,y\}})_{y\in V\backslash \{x\}}$ is summable.
\end{remark}
Thus, as a consequence of the assumptions in (\ref{ExistenceAssumption}) the probability that a long edge is open, i.e.~an edge connecting two vertices over a long distance, becomes exceedingly small. Broadly speaking this means that a successful infection over a long distance is getting more and more unlikely as the distance increases. The second part of the assumption can be seen as assuming that all edges attached to an arbitrary vertex are updated after a finite time. This might seem a bit unintuitive, but we need this assumption for technical reasons. We discuss the necessity of this rate assumption briefly in Section~\ref{Discussion&OpenProblems} right after Problem~\ref{Prob:SurvivalForConnectedDymGraphs}.

In Section \ref{GraphRepCPDLP}  we will explicitly construct the CPDLP via a graphical representation and then show that under these assumptions the resulting process is in fact a well-defined Feller process (see Proposition \ref{propn:Feller}) with state space $\cP(V)\times \cP(\cE)$ and that $|\bfC_t|<\infty$ almost surely for all $t\geq 0$ if $|\bfC_0|<\infty$, even if the background is started in $\cE$ (see Proposition \ref{propn:finiteinfectedsets2}).

\smallskip
We are interested in the survival behavior of the CPDLP as the parameter $\lambda$ varies, and later on also as $\gamma>0$ and $q\in(0,1)$ vary for percolation and speed kernels of the form (\ref{scaledkernels}). In the general setting, we denote by
\begin{equation*}
	\theta(\lambda,C):=
	\Pw(\bfC^{C}_t\neq\emptyset\,\, \forall t\geq0)
\end{equation*}
the \emph{survival probability} of a CPDLP with infection parameter $\lambda$ and initial state $\bfC_0=C$ and $\bfB_0\sim \pi$ (and all other parameters fixed).

We denote the \emph{critical infection rate} for survival by \vspace{-1mm}
\begin{align*}
	\lambda_c:=\inf \{\lambda\geq 0: \theta(\lambda,\{x\})>0\},
\end{align*}
where $x\in V$ is chosen arbitrary. Note that by translation invariance it follows that 
$\theta(\lambda,\{x\})=\theta(\lambda,\{y\})$ for all $x,y\in V$. Furthermore, this together with the additivity of the infection process $\bfC$ implies that $\theta(\lambda,C)>0$ for some $C\subset V$ with $0<|C|<\infty$ then this is true for all such sets. 
Thus, the definition of the critical rate $\lambda_c$ does not depend on the choice of the set of initially infected vertices as long as it is finite and non-empty.

Furthermore, by standard methods and using the monotonicity of the system, see also Remark~\ref{BasicProperties}, we get the existence of the upper invariant law  $\overline{\nu}$, which is the weak limit of the process started with $(\bfC_0,\bfB_0)=(V,\cE)$. (Whenever relevant we will indicate the dependence of $\overline{\nu}$ on the parameters of the model with subscripts.)
The upper invariant law is the largest invariant law according to the stochastic order, i.e. if $\nu$ is an invariant law of the CPDLP, then $\nu\preceq \overline{\nu}$, where "$\preceq$" denotes the stochastic order. Of course, there exists the trivial invariant law $\delta_{\emptyset}\otimes \pi$. This poses the question for which parameters 
it holds that $\overline{\nu}=\delta_{\emptyset}\otimes \pi$. Since it is not difficult to see that $\delta_{\emptyset}\otimes \pi$ is the smallest invariant law possible, $\overline{\nu}=\delta_{\emptyset}\otimes \pi$ is equivalent to 
ergodicity of the system, i.e.~that there exists a unique invariant law which is the weak limit of the process. We define the critical infection rate for non-triviality of the upper invariant law by 
\begin{equation*}
	\lambda'_c:=\inf\{\lambda>0:\overline{\nu}_{\lambda}\neq\delta_{\emptyset}\otimes \pi\}.
\end{equation*}

\section{Main results}
Our first result is that the critical infection rate of survival and the critical infection rate for a non-trivial  upper invariant law are the same. 
\begin{theorem}\label{CrticalValuesAgree}
	The two critical infection rates coincide, i.e.\ $\lambda'_c=\lambda_c$.
\end{theorem}
The next result provides a coupling of the CPDLP with a contact process that has a general infection kernel. As a consequence we obtain that if this contact process survives then this implies survival of the CPDLP, and thus this leads to a sufficient criterion for a positive survival probability of the CPLDP. We first define the contact process $\bfX$ with an \emph{infection kernel} $(a_{e})_{e\in \cE}\subset[0,\infty)$ and recovery rate $r>0$ on the complete graph $(V,\cE)$. We additionally assume that $a_{\{x,y\}}=a_{\{x',y'\}}$ if $d(x,y)=d(x',y')$ and
\begin{equation}
	\label{RateBound}
	\sum_{y\in V\backslash\{x\}}a_{\{x,y\}}<\infty
\end{equation}
for all $x\in V$. If $\bfX$ is currently in the state $C$ it has the transitions
\begin{equation}
	\label{GeneralInfection}
	\begin{aligned}
		C\to C\cup \{x\}	\quad &\text{ at rate } \sum_{y\in C} a_{\{x,y\}} \text{ and }\\
		C\to C\backslash \{x\}	\quad &\text{ at rate } r.
	\end{aligned}
\end{equation}
This process can again be constructed via a graphical representation and it is a well known fact that if \eqref{RateBound} is satisfied then $\bfX$ is a well-defined Feller process  on the state space $\cP(V)$, see for example \cite[Propostion~I.3.2]{liggett2012interacting} and \cite{swart2009contact}.
As usual we indicate the initial configuration $C\subset V$ by adding a superscript $\bfX^C$, i.e.\ $\bfX^{C}_0=C$. 
\begin{theorem}
	\label{ComparisonWithLongRangeCP}
	Let $C\subset V$ and $(\bfC^{C}_t,\bfB_t)_{t\geq 0}$ be a CPDLP with parameter $\lambda,q$ and $\gamma$. Then there exists a contact process  $(\overline{\bfX}^{C}_t)_{t\geq0}$ with $\overline{\bfX}^{C}_0=C$ and with infection kernel 
	\begin{align*}
		\overline{a}_{e}(\lambda):=\frac{1}{2}\Big(\lambda+\hat{v}_{e}-\sqrt{(\lambda+\hat{v}_{e})^2- 4 \lambda\hat{v}_{e}\hat{p}_{e}}\,\Big)\geq 0
	\end{align*}
	for all $e\in \cE$ such that $\overline{\bfX}^{C}_t\subset \bfC^{C}_t$ for all $t\geq 0$. Thus, in particular
	$$\lambda_c	\leq \overline{\lambda}_c,$$
	where $\overline{\lambda}_c:=\inf\{\lambda>0:\Pw(\overline{\bfX}^{\{x\}}_t\neq \emptyset \,\,\forall t\geq 0)>0\}$ is the critical infection rate for survival of $\overline{\bfX}$ (which is again independent of $x \in V$).
\end{theorem}
The following results are concerned with the behavior of survival as we scale percolation probability and speed with the parameters $q$ and $\gamma$. Thus, we assume that the percolation and speed kernels are as in (\ref{scaledkernels}) and consider the survival probability $\theta$ and the critical infection parameter $\lambda_c$ as functions of $\gamma$ and $q$.
\begin{remark}
	Note that in this setting a CPDLP with rates $\lambda,r,\gamma,q$ has the same dynamics as a CPDLP with rates $\lambda/r,1,\gamma/r,q$ when time is sped up by a factor of $r$, and the survival probabilities are the same.
	Thus, the survival behavior of a CPDLP with rates $\lambda,r,\gamma,q$ 
	can be deduced from the survival behavior of a CPDLP with rates $\lambda,1,\gamma,q$.
	In other words, it is not necessary to explicitly study the dependence on $r$.
\end{remark}

As a consequence of Theorem \ref{ComparisonWithLongRangeCP} we obtain the following result for fast update speed.
\begin{corollary}\label{CritRateFastSpeed}
	Let $\bfX^{\infty}$ be a contact process with infection kernel $(\lambda q p_e)_{e\in \cE}$ and denote the corresponding critical infection rate by
	\begin{equation*}
		\lambda^{\infty}_{c}(q)=\inf\{\lambda>0:\Pw_{\lambda,q}(\bfX^{\{x\},\infty}\neq \emptyset \,\, \forall t\geq 0)>0\}.
	\end{equation*}
	Then we have
	\begin{equation*}
		\limsup_{\gamma\to \infty} \lambda_c(\gamma,q)\leq\lim_{\gamma\to \infty}\overline{\lambda}_c(\gamma,q)=\lambda^{\infty}_{c}(q)<\infty.
	\end{equation*}
\end{corollary}

The next result shows that for any fixed speed if we have overall a low probability that an edge of any length is open, i.e. for $q$ small enough, we are  in the immunization region for the CPDLP. This means that the critical infection rate is infinite and so no matter how large the infection rate is the CPDLP will die out almost surely 
\begin{theorem}\label{ImmunThm}
	For any fixed $\gamma>0$ there exists $q_0=q_0(\gamma)\in (0,1]$ such that $\bfC$ dies out almost surely for all $q <q_0$, regardless of the choice of $\lambda>0$, i.e.~$\lambda_c(\gamma,q)=\infty$ for all $q<q_0$,  and such that $\lambda_c(\gamma,q)<\infty$ for all 
	$q>q_0$. Moreover, the function $\gamma \mapsto q_0(\gamma)$ is monotone non-increasing on $(0,\infty).$
\end{theorem}
As a corollary we can also get more insight into the  behavior of  the critical infection rate $\lambda_c(\gamma,q)$
as a function of $\gamma$ and in particular into its asymptotic behavior 
as $\gamma \to 0$.
Note that while it is clear due to monotonicity that $q \mapsto \lambda_c(\gamma,q) $ is a monotone non-increasing function, see also Remark~\ref{BasicProperties}, this is not so clear for the function $\gamma \mapsto \lambda_c(\gamma,q)$. Nonetheless, it can be shown (see Proposition~\ref{WeakSpeedMonotonicity}) that $\lambda_c(\gamma,q)$ can at most increase linearly in $\gamma$ which implies that 
there exists a $\gamma_0(q)$ such that $\lambda_c(\gamma,q)=\infty$ for all $\gamma<\gamma_0(q)$ and $\lambda_c(\gamma,q)<\infty$ for all  $\gamma>\gamma_0(q)$. Note that $\gamma_0(q)$ must be  finite for any $q>0$ due to Corollary \ref{CritRateFastSpeed} but that it may be $0$.
However, the following corollary states that for small enough $q$ we have $\gamma_0(q)>0$ such that a nontrivial immunization phase exists.
For this we now set
\begin{equation}
	\label{def:q_1}
	q_1:=\sup_{\gamma \in (0,\infty)} q_0(\gamma)= \lim_{\gamma \rightarrow 0} q_0(\gamma),
\end{equation}
where we have used  the monotonicity of $\gamma \mapsto q_0(\gamma)$ stated in Theorem \ref{ImmunThm}.
This means that by Theorem \ref{ImmunThm} for every $q< q_1$ there exists a 
$\gamma>0$ such that $q<q_0(\gamma) <q_1$ which implies $\lambda_c(\gamma,q)=\infty$ and thus also $\gamma_0(q)>0$. In summary, we have the following statement:

\begin{corollary}\label{PartialResultsSlowSpeed}
	For every $q\in (0,1]$ there exists a $\gamma_0=\gamma_0(q)\in [0,\infty)$ 
	such that $\lambda_c(\gamma,q)=\infty$ for all $\gamma<\gamma_0$  and $\lambda_c(\gamma,q)<\infty$ for all $\gamma>\gamma_0$.
	Furthermore, there exists a $q_1\in(0,1]$, see \eqref{def:q_1}, so that for every $q< q_1$ 
	we have $\gamma_0(q)>0$ while  for every $q> q_1$ we have $\gamma_0(q)=0.$
	This implies in particular for every $q< q_1$ that $\lim_{\gamma\to0}\lambda_c(\gamma,q)=\infty$.
\end{corollary}

For general countable vertex sets $V$ we can only determine that $\lambda_c(\gamma,q) \rightarrow \infty$  when $\gamma  \rightarrow 0$ if $q<q_1$ is small enough. But in the special case $V=\Z$ and $E=\{\{x,y\}\subset \Z: |x-y|=1\}$, i.e.~when $G=(V,E)$ is the $1$-dimensional integer lattice we can conclude that this is the case for all $q<1,$ under some further assumptions. In fact, these assumptions even guarantee that for any fixed $\lambda$ we cannot have survival  if the update speed $\gamma$ is small enough.

\begin{theorem}\label{AsymptoticSlowSpeedThm}
	Consider $G$ to be the $1$-dimensional integer lattice. Let $q<1$ be fixed and $C\subset V$ be non-empty and finite. Furthermore, assume that the sequences $(p_{e})_{e\in \cE}$ and $(v_{e})_{e\in \cE}$ satisfy 
	\begin{equation}\label{StrongerAssumption}
		\sum_{y\in \N}yv_{\{0,y\}}p_{\{0,y\}}<\infty
		\quad \text{ and } \quad 
		\sum_{y\in \N}yv_{\{0,y\}}^{-1} <\infty.
	\end{equation}
	Then, for every $\lambda>0$ there exists $\gamma^*=\gamma^*(\lambda,q)>0 $ such that $\bfC^{C}$ dies out almost surely for all $\gamma\leq\gamma^*$, i.e.~$\theta(\lambda,\gamma,q,C)=0$ for all $\gamma\leq\gamma^*$. Thus, in particular $\lim_{\gamma\to 0}\lambda_c(\gamma,q)=\infty$.
\end{theorem}
Note that \eqref{StrongerAssumption} is a stronger assumption than \eqref{ExistenceAssumption}, and thus already implies the latter assumption.

\subsection{Outline}
The rest of this paper is organized as follows. In Section~\ref{Discussion&OpenProblems} we discuss some related literature in order to put our results into context with the current state of research. Then, we state and discuss some open problems and possible directions for future research.

Since we will use on several occasions a comparison with an independent long range percolation model we introduce this type of model in Section~\ref{LongRangePercolationModel} and state some conditions which imply the absence of an infinite connected component.

In Section~\ref{GraphRepCPDLP} we construct the CPDLP via a graphical representation. Furthermore, in Subsection~\ref{Welldefined&Feller} we show that this construction yields a well-defined Feller process. In Subsection~\ref{InvariantLaw} we describe the construction of a dual infection process, which yields a self-duality relation. We then use this relation to prove Theorem~\ref{CrticalValuesAgree}.
In Subsection~\ref{ConsequencesOfTheGraphRep} we use the graphical representation to show Theorem~\ref{ComparisonWithLongRangeCP} and Corollary~\ref{CritRateFastSpeed}.

In Section~\ref{BlockwiseLongRangePercolation} we compare the dynamical long range  percolation blockwise with an independent long range percolation model and define a new infection process, which dominates the original one. We use this newly defined process to show Theorem~\ref{ImmunThm} in Subsection~\ref{ImmunizationRegime}. Lastly, we show 
Theorem~\ref{AsymptoticSlowSpeedThm} in Subsection~\ref{SlowSpeedRegime}.

\section{Discussion}\label{Discussion&OpenProblems} 
\subsection{Related literature}
\label{sec:relatedlit}
The contact process was first introduced almost half a century ago by Harris \cite{harris1974contact} on $\Z^d$. Since then this process and many variations of it have been studied intensively, mostly on bounded degree graphs. 
To the best of our knowledge the first to introduce  a long range variation of the contact process, where there is no intrinsic bound on the distance between two vertices for which a transmission of an infection can take place,  was Spitzer~\cite{spitzer1977stochastic}. He studied so-called nearest particles systems. Bramson and Gray~\cite{bramson1981note} studied in particular the phase transition of similar systems. See also \cite[Chapter~VII]{liggett2012interacting} for more results on nearest particles systems.

Swart~\cite{swart2009contact} studied a contact process with general infection kernel $(a_e)_{e\in \cE}$ as in \eqref{RateBound} and \eqref{GeneralInfection}, see also \cite{athreya2010survival}, \cite{sturm2014subcritical} and \cite{swart2018simple} for more results on this process. For applications in certain areas of physics, see for example \cite{ginelli2006contact}.

Another long range variation of the model 
is a contact process defined on a random graph with unbounded degree. To be precise, the considered graph has almost surely finite degree but there exists no uniform bound for the degree of a vertex. 
For example, Can \cite{can2015contact} studied a contact process on an open cluster generated by a long range percolation, and M\'enard and Singh \cite{menard2015percolation} considered the phase transition of contact processes on more general graphs of unbounded degree.

In this paper we  study the spread of an infection in a dynamical random environment. To the best of our knowledge the first to study such a model  explicitly was Broman \cite{broman2007stochastic} followed by Steif and Warfheimer \cite{steif2007critical}, who considered a contact process with varying recovery rates. 
Remenik \cite{remenik2008contact} studied a related model and made connections to multi-type contact processes, which had  been studied earlier,
see for example Durrett and M\o ller \cite{durrett1991complete}.

Linker and Remenik \cite{linker2019contact} explicitly considered a contact process on a dynamical percolation, i.e.~in the evolving random environment the edges of the underlying graph of bounded degree open and close independently, and the closed edges cannot be used by the infection. They studied the phase transition for survival.
As a follow up we \cite{seiler2022contact} studied a contact process in a more general evolving edge random environment, but still kept the assumption that the underlying graph has bounded degree. For more work in this direction, see also Hilario et al.~\cite{hilario2021results}. In the present article we consider a graph with unbounded degree as a natural long range extension to these models on bounded degree graphs.

Finally we mention recent work by Gomes and Lima \cite{gomes2021long} on the survival probability of a dynamical long range contact process, which in contrast to our model includes a vertex update mechanism. At update events a vertex independently has a radius assigned according to some distribution. From this time on this vertex can infect every neighbor inside the ball of this radius. Note that in this model the orientation of the edges is important while in our model this is not the case since an edge is either open or closed for infections in either direction.

\subsection{Discussion and open problems}
Theorem~\ref{ComparisonWithLongRangeCP} states that a contact process $\overline{\bfX}$ with a specific infection kernel acts as a bound from below for the CPDLP such that survival of $\overline{\bfX}$ implies survival of the CPDLP. For this result we use a comparison result developed by Broman \cite{broman2007stochastic}. Furthermore, in Corollary~\ref{CritRateFastSpeed} we show that the critical infection rate $\lambda^{\infty}_{c}(q)$ of a contact process $\bfX^{\infty}$ with infection kernel $(\lambda q p_e)_{e\in \cE}$ is an upper bound for the limit of the critical infection rate of the CPDLP, i.e. 
\begin{equation*}
	\limsup_{\gamma\to \infty} \lambda_c(\gamma,q)\leq\lim_{\gamma\to \infty}\overline{\lambda}_c(\gamma,q)=\lambda^{\infty}_{c}(q)<\infty.
\end{equation*}
But in fact we will see in Lemma \ref{AsympFastSpeedLimitRates} that the rates $\overline{a}_e(\lambda, q,\gamma)$ of $\overline{\bfX}$ converge to $\lambda q p_e$ as $\gamma\to \infty$ from below. Thus, it seems plausible to assume that the following conjecture holds true.
\begin{conjecture}\label{LimitFastSpeed}
	Fix $q$ and assume that \eqref{ExistenceAssumption} is satisfied.  Then we conjecture that
	\begin{equation*}
		\lim_{\gamma\to \infty}\lambda_c(\gamma,q)= \lambda^{\infty}_{c}(q).
	\end{equation*}
\end{conjecture}
The shape of the infection kernel $(\lambda q p_e)_{e\in \cE}$ does have a heuristic explanation. Let us consider a particular edge $e$ and a given infection rate $\lambda$. If the update speed $\hat{v}_{e}$ is chosen significantly larger than the infection rate, i.e.~$\gamma$ large enough, then with high probability there will be an update event between two consecutive infection events, and thus this results heuristically speaking in a thinning of the infection process such that infection events take place at rate $\lambda \hat{p}_e$ as $\gamma\to \infty$. Linker and Remenik made this heuristic rigorous in the proof of \cite[Theorem~2.3]{linker2019contact}. If we could extend their proof to our model we would have shown Conjecture~\ref{LimitFastSpeed}. But several steps of their proof rely heavily on the fact that they only consider graphs with bounded degrees.
Nevertheless we believe that since we additionally assume that $\hat{v}_{\{x,y\}}\to \infty$ as $d(x,y)\to \infty$ it should be possible to make this heuristic argument into a rigorous proof. 

\smallskip
In Theorem~\ref{ImmunThm}, which is analogous to the result \cite[Theorem~2.6(a)]{linker2019contact} in the bounded degree case, we prove the existence of the so called immunization phase if the sequences $(v_e)_{e\in \cE}$ and $(p_e)_{e\in \cE}$ satisfy \eqref{ExistenceAssumption}. This means that for a given $\gamma$ there exists $q_0=q_0(\gamma)\in(0,1]$ such that $\lambda_c(\gamma,q)=\infty$ for all $q <q_0$ and 
$\lambda_c(\gamma,q)<\infty$ for all $q >q_0$.
In other words, if the parameter $q$  of the background dynamics is chosen small enough no survival is possible no matter how large the infection rate $\lambda$ is. Linker and Remenik \cite{linker2019contact} showed for the contact process on a one-dimensional nearest neighbor percolation that a particular threshold value $p_1 \in (0,1)$ exists such that an immunization phase exists if $p<p_1$ and that no such phase exist for $p>p_1$. Thus, by a comparison argument, if there exists an edge $e\in \cE$ such that $p_{e}$ is larger than this threshold $p_1$ then also $q_0(\gamma)<1$.
The value $q_1$ of Corollary~\ref{PartialResultsSlowSpeed} was defined 
as the supremum over all values of $q_0(\gamma)$ such that for all $q>q_1$ we have $\lambda_c(\gamma,q)<\infty$ for all $\gamma>0$. This leads to the following question:
\begin{openproblem}
	For which sequences $(p_e)_{e\in \cE}$ and $(v_e)_{e\in \cE}$  do we have $q_1<1$ so that there is a nontrivial phase transition in the parameter $q$?
\end{openproblem}
It is clear from \cite{linker2019contact} that $q_1=1$ in the nearest neighbor setting, namely if $p_{\{x,y\}}<p_1$ for $d(x,y)=1$ (and equal to zero otherwise). In a truly long range setting ($p_e>0$ for all $e\in \cE$) 
the proof of Theorem~\ref{ImmunThm} suggests that if $\sum_{y\in V\backslash\{x\}}p_{\{x,y\}}$ is small enough it should hold that $q_1=1$.

\smallskip
As a direct consequence of Theorem~\ref{ImmunThm} we are able to characterize the survival behavior for small $q$. To be precise Corollary~\ref{PartialResultsSlowSpeed} yields that the infection will almost surely die out as $\gamma \to 0$, i.e~$\lim_{\gamma\to 0} \lambda_c(\gamma,q)=\infty$ for $q<q_1$. In the special case $V=\Z$ with Theorem~\ref{AsymptoticSlowSpeedThm}, which is analogous to the result \cite[Theorem~2.4(a)]{linker2019contact} in the bounded degree case, see also \cite[Theorem~1.1(i)]{hilario2021results} for higher dimensions,   
we are able to show that $\lim_{\gamma\to 0} \lambda_c(\gamma,q)=\infty$ for all $q<1$ if we sharpen the assumptions on $(v_e)_{e\in\cE}$ and $(p_e)_{e\in\cE}$ from \eqref{ExistenceAssumption} to \eqref{StrongerAssumption}.
As we will see in Section~\ref{SlowSpeedRegime} these sharper assumptions are crucial in the proof, where we use a comparison to a long range percolation model. The assumption $\sum_{y\in \N}yv_{\{0,y\}}p_{\{0,y\}}<\infty$ implies in particular that this long range percolation model has no infinite component. 
Thus, for $\Z$ as well as more general graphs a natural question is what the asymptotic behavior is if we assume that the long range percolation induced by $(qp_{e})_{e\in \cE}$ forms a infinite connected component? In this case, the background process would contain an infinite connected component at any time point $t$. Then we expect that this should imply the possibility of survival for any background speed, i.e. that $\lim_{\gamma\to 0} \lambda_c(\gamma,q)=\infty$ does not hold.

\begin{openproblem}\label{Prob:SurvivalForConnectedDymGraphs}
	Assume that  there exists a $q\in(0,1)$ such that the long range percolation induced by $(q p_{e})_{e\in \cE}$ forms an infinite connected component almost surely and  let  $q^*\in[0,1)$ denote the infimum over all such $q$.
	Do we have    for all $q>q^*$ that
	$\sup\{\lambda_c(\gamma,q'):\gamma \geq 0, q'\in(q,1)\}<\infty$?
\end{openproblem}
In \cite[Theorem 1.1(ii)]{hilario2021results} this was shown in the special case of the  contact process on a dynamical nearest neighbor percolation defined on the $d$-dimensional integer lattice.

\smallskip
Let us mention that the assumption $\sum_{y\in V\backslash \{x\}}v_{\{x,y\}}^{-1}<\infty$ for all $x\in V$ in \eqref{ExistenceAssumption} implies that every edge $\{x,y\}$ attached to an arbitrary vertex $x$ is almost surely updated after a finite time. Heuristically speaking this means that the neighborhood of $x$ is "reset" after a finite time almost surely, i.e.\ all edges attached to $x$ are updated at least once.
This assumption is necessary for the proof strategy of Theorem~\ref{ImmunThm} and Theorem~\ref{AsymptoticSlowSpeedThm}. 
But the assumption also means that the update speed of an edge tends to infinity as the length of the edge grows, i.e.~$v_{\{x,y\}}\to \infty$ as $d(x,y)\to \infty$ while it would seem more natural to assume that $v_e$ is constant for all $e\in \cE$. However, in the case of constant speed we would have to restrict the state space as the CPDLP will not be a well-defined Feller process on the full state space $\cP(V)\times\cP(\cE)$. 
This is because  if we start with every edge $e$ in the state open, i.e.~$\bfB_0=\cE$ the infection process will explode in finite time since any neighborhood will contain infinitely many neighbors almost surely.
Nevertheless one could consider this process with constant speed, i.e.\ $v_e=v$ for all $e$, on a smaller state space. A possible choice for such a state space would be the set of edge configurations which are locally finite.

Finally, one may investigate whether $\lambda_c=0$ may be possible for certain parameter regimes.
For contact processes on graphs with bounded degrees it is easy to see that a subcritical phase exists, i.e.  $\lambda_c>0$, 
which can be proven via a comparison with a continuous time branching process with binary offspring distribution. In general, such a comparison can also be done for contact processes on the complete graph with a summable infection kernel $(a_e)_{e\in \cE}$. The procedure is similar to that used in the proof of Proposition~\ref{lowerBound&integrability} where we use such a comparison to show existence of a non-trivial phase transition for a long range percolation model.

On the other hand if we consider a contact process with constant infection rate $\lambda$ on a locally finite random graph with unbounded degree this kind of a comparison might fail, and in fact such systems do not always exhibit a phase transition. For example, Gomes and Lima manage to show this for their model, see \cite[Theorem~2]{gomes2021long}. They show that if 
the typical radius of the region of vertices which may be infected is 
large enough, then the infection survives for any infection rate $\lambda$.
See also \cite{can2015contact}, \cite{menard2015percolation} and \cite{huang2020contact}, where this type of question is studied for a contact process on  static random graphs. For our model we have the following conjecture:
\begin{conjecture}\label{SubCriticalRegimeConj}
	If $(v_e)_{e\in\cE}$ and $(p_e)_{e\in\cE}$ satisfy \eqref{ExistenceAssumption} then $\lambda_c(\gamma,q)>0$ for all $\gamma>0$ and $q\in[0,1]$. In words, this means that for all choices of $\gamma$ and $q$ there exists a subcritical phase for the infection.
\end{conjecture}
Our intuition comes from a recent work by
Jacob, Mörters and Linker who consider in \cite{jacob2022contact} a related setting to ours but on finite graphs of size $N$. In \cite[Section~5]{jacob2022contact} they define an auxiliary process which they call the \emph{wait-and-see} process which dominates the infection process, and they show with a supermartingale argument that under some conditions the process has a fast extinction regime if the infection rate $\lambda$ is small enough, i.e.\ the extinction time of the infection process is bounded by some power of $\log(N)$. We believe that these techniques can be adjusted to infinite graphs in such a way that they would imply Conjecture~\ref{SubCriticalRegimeConj}.

\section{Long range percolation model}\label{LongRangePercolationModel}
Several proofs rely on  a comparison argument with a long range percolation model. Thus, in this section we briefly introduce this model and show two results concerning the absence of an infinite connected component. One of the first papers to mention a long range percolation model is by Schulman \cite{schulman1983long}. There, Proposition~\ref{lowerBound&integrability} and Proposition~\ref{CutPointTheo} are shown in special cases. Since we could not find these results in the literature in the generality that we need, we prove these results in this section, for the sake of completeness.

The long range percolation model $w$ takes values in 
$\bigotimes_{e\in \cE}\{0,1\}$ such that $w=(w(e))_{e\in \cE}$ is a family of independent random variables with  $\Pw(w(e)=1)= b_e\in[0,1]$ for all $e\in \cE$.
We declare an edge $e=\{x,y\}\in \cE$ to be open if $w(e)=1$. 
We assume for every fixed $x\in V$ that $\sum_{y\in V\backslash\{x\}}b_{\{x,y\}}<\infty$ to guarantee that $\big(V,w^{-1}(\{1\})\big)$  is a locally finite graph, where $w^{-1}(\{1\})=\{e\in \cE:w(e)=1\}$. Furthermore, we again assume translation invariance, i.e. that $b_{\{x,y\}}=b_{\{x',y'\}}$ if $d(x,y)=d(x',y')$, where $d(\cdot,\cdot)$ is the graph distance induced by $G=(V,E)$. We denote by $\cC(x)$ the connected component of $w$ containing $x\in V.$ The following result provides a sufficient condition for absence of percolation.
\begin{proposition}\label{lowerBound&integrability}
	Let $\sum_{y\in V\backslash\{x\}}b_{\{x,y\}}<1$ for one, and hence every $x\in V$. Then almost surely there exists no infinite connected component. In this case $|\cC(x)|$ is also integrable for all $x\in V$.
\end{proposition}  
\begin{proof}
	This can be proven via a coupling with a branching process. Since $V$ is countable we can enumerate all vertices such that $V=\{x_0,x_1,\dots\}$.  
	We denote the set of all paths of length $n$ starting at $x_0$ by 
	\begin{equation*}
		\cT:=\{(\alpha_0,\alpha_1,\dots,\alpha_n)\in V^{n+1}: n\in\N_0, \alpha_0=x_0 \text{ and } \alpha_{i-1}\neq \alpha_{i} \text{ for all } 0<i\leq n  \}.
	\end{equation*}
	For $\alpha=(\alpha_0,\alpha_1,\dots,\alpha_n)$ we define the generation of $\alpha$ as $|\alpha|=n$ (so that $|(x_0)|=0$). Furthermore, we equip $\cT$ with the lexicographical order with respect to the enumeration of $V$.
	
	Now we construct a family of random variables $(X_{\alpha})_{\alpha\in \cT}$ with $X_{\alpha}\in \{0,1\}$. We set  $X_{(x_0)}=1$ and define the remaining $X_{\alpha}$ successively with increasing $|\alpha|$. 
	When defining $X_{\alpha}$ for a given $|\alpha|$ we proceed in lexicographical order, which also means that we 
	"visit" the parents with $|\alpha|-1$ in lexicographical order to define the values for their children.
	We want $Z_{n}:=\sum_{\alpha\in \cT:|\alpha|= n}X_{\alpha}$ to define a branching process with $Z_0=1$. We also want to ensure that for any $x \in \cC(x_0)$ there is an $n \in \N_0$ and $\alpha\in \cT$ with $|\alpha|=n$
	and $\alpha_n=x$ such that $X_{\alpha}=1,$  which implies that $|\cC(x_0)|\leq \sum_{n\in \N}Z_n:=T$, where $T$ is the total progeny of the branching process. (However, an $x \in \cC(x_0)$ can appear multiple times in the 
	branching tree.)
	
	As part of the construction we also successively define increasing sets $I_{\alpha}\subset \cE$ starting with 
	$I_{(x_0)}= \emptyset$. These sets contain all the edges that we have used in the definitions prior to determining $X_{\alpha}$.
	More precisely, suppose we have already constructed all $X_{\alpha'}$ before now defining $X_{\alpha}$ with $\alpha=(\alpha_0,\alpha_1,\dots,\alpha_n).$
	Then, $I_{\alpha}$ contains all edges $\{y,z\}\in \cE$ for which there exists an $\alpha'\in \cT$ with $X_{\alpha'}=1$ such that $|\alpha'|=k< n-1$ and $\alpha'_{k}=y$ or such that $|\alpha'|= n-1$,  $\alpha'_{n-1}=y$ and $\alpha'$ smaller than $(\alpha_0,\alpha_1,\dots,\alpha_{n-1})$ in lexicographical order. 
	We let $w^{\alpha}$ be an independent copy of $w$ for all $\alpha \in \cT$.
	Now, we define $X_{\alpha}$ by
	\begin{align*}
		X_{\alpha}:=
		\begin{cases}
			1 & \text{ if } \{\alpha_{n-1},\alpha_n\}\notin I_{\alpha},\, w(\{\alpha_{n-1},\alpha_n\})=1 \text{ and } X_{(\alpha_0,\dots,\alpha_{n-1})}=1,\\
			1 & \text{ if } \{\alpha_{n-1},\alpha_n\}\in I_{\alpha},\, w^{\alpha}(\{\alpha_{n-1},\alpha_n\})=1 \text{ and } X_{(\alpha_0,\dots,\alpha_{n-1})}=1,\\
			0 & \text{ otherwise. }
		\end{cases}
	\end{align*} 
	This in particular implies that $X_{\alpha'}=0$ for any descendant of an $\alpha$ with $X_{\alpha}=0.$ Also, $X_{\alpha}=1$ is only possible if $\alpha_{|\alpha|} \in \cC(x_0)$. On the other hand, for any 
	$x  \in \cC(x_0)$ there will be a $X_{\alpha}=1$ with $\alpha_{|\alpha|}=x.$
	In order to get independence between different generations and between the several offspring of the same generation we used independent copies $w^{\alpha}$ instead of $w$. 
	In words if we have that $X_{(\alpha_0,\dots,\alpha_{n-1})}=1$ and we have not used $w(\{\alpha_{n-1},\alpha_n\})$ yet 
	($\{\alpha_{n-1},\alpha_n\} \notin I_{\alpha}$) then we set $X_{\alpha}=1$ if $w(\{x_m,\alpha_n\})=1$, otherwise we use the 
	independent $w^{\alpha}(\{\alpha_{n-1},\alpha_n\})$.
	This is why $Z=(Z_n)_{n\in \N_0}$ with $Z_{n}:=\sum_{\alpha\in \cT:|\alpha|= n}X_{\alpha}$ now defines a branching process for which the offspring distribution is the same in every step because of translation invariance. 
	In particular, the offspring mean is given by $\mu:=\sum_{y\in V\backslash\{x\}}b_{\{x,y\}}$.

	It is well known that for $\mu<1$ the branching process dies out almost surely which provides the first claim. It also holds that $\E[T]\leq\frac{1}{1-\mu}$ for $\mu<1$ as for example shown in \cite[Theorem 3.5]{van2016random}, which provides integrability of $|\cC(x_0)|$. Because of translation invariance this result does not depend on the choice of $x_0$ since $|\cC(x_0)|\stackrel{d}{=}|\cC(y)|$ for all $y\in V$. 
\end{proof}
Next we consider the special case $V=\Z$ and $E=\{\{x,y\}\subset \Z: |x-y|=1\}$. Since we assume translation invariance we can simplify notation and set $b_{\{n,n+k\}}=b_{\{0,k\}}=:b_k$ for all $k\in \N$ and all $n\in \Z$. Here, an infinite component can only exist if $\sum_{k\in \N} kb_k=\infty$. The reason for this is that if $\sum_{k\in \N} kb_k<\infty$ then the long range percolation is similar to a finite range percolation in the sense that there appear so-called ``cut-points", see Figure~\ref{fig:CutPoint}, which lead to a partition of the integer lattice $\Z$ into finite connected components. We will briefly show this result for the long range percolation before we continue with our study of the CPDLP.

\begin{definition}\label{CutPoint}
	Let $V=\Z$. A \textit{cut-point} $m\in \Z$ is a point such that no edge $\{x,y\}$ with $x\leq m<y$ is present in the model, i.e.~$\omega(\{x,y\})=0$.
\end{definition}
\begin{figure}[!b]
	\centering 
	\includegraphics[width=90mm]{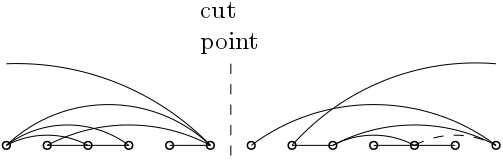}
	\caption{Illustration of a cut point.}
	\label{fig:CutPoint}
\end{figure}
In the proof of the following result ergodic theory is used. We give a brief summary of some of the important notions.
Let $(\Omega,\cF,\Pw)$ be a probability space and $S:\Omega\mapsto\Omega$ be a measure-preserving map, i.e.~$\Pw^{S}=\Pw$. We denote by $\cI = \{A \in  \cF : A=S^{-1}(A) \}$ the invariant $\sigma$-algebra. We call 
$(\Omega,\cF,\Pw,S)$ an ergodic system if $\cI$ is $\Pw$-trivial, i.e.~if $A\in \cI$, then $\Pw(A)\in \{0,1\}$. Let $X$ be the identity on $\Omega$, i.e.~$X(\omega)=\omega$ and $f:\Omega\to\R$ be a measurable function. The mean ergodic theorem of Birkhoff, see \cite[Theorem~9.6]{kallenberg2006foundations}, in particular states that if $(\Omega,\cF,\Pw, S)$ 
is ergodic then
\begin{align}
	\label{Birkhoff}
	\frac{1}{2n}\sum_{k=-n}^{n}f(S^{-k}X)\to \E[f(X)]\quad \text{ as } n\to\infty.
\end{align}

\begin{proposition}\label{CutPointTheo}
	Let $(b_k)_{k\in \N}\subset [0,1)$ with $\sum_{k\in \N} kb_k<\infty$. Then the following holds:
	\begin{enumerate}
		\item For $m\in \Z$ the probability $\Pw(m \text{ is a cut-point})=\Pw(0 \text{ is a cut-point})>0$, and as a consequence there exist almost surely infinitely many cut-points.
		\item The subgraphs induced in the intervals between consecutive cut-points are independent and identically distributed. In particular, this implies that the distances between consecutive cut-points form a sequence of i.i.d.\ random variables as well.
		\item Almost surely there exists no infinite connected component.
	\end{enumerate}
\end{proposition}
\begin{proof}
	By translation invariance we know that
	\begin{equation*}
		\Pw(m \text{ is a cut-point})=\Pw(0 \text{ is a cut-point})=\prod_{x\leq 0<y} (1-b_{\{x,y\}}).
	\end{equation*}
	The infinite product on the right hand side is strictly positive since 
	\begin{equation*}
		\sum_{x\leq 0<y} b_{\{x,y\}}=\sum_{k=1}^{\infty} \sum_{l=0}^{k-1} b_{\{-l,-l+k\}} =\sum_{l\in \N}k b_{k}<\infty,
	\end{equation*}
	where we used that $b_{\{-l,-l+k\}}=b_{k}$ for every $l\in \Z$. Thus, this yields the first claim. Next let us define $X_m:=\1_{\{m \text{ is a cut-point}\}}$. Let $S$ be a shift operator on $\Omega=\bigotimes_{e\in \cE}\{0,1\}$  such that
	\begin{equation*}
		(\omega(\{x,y\}))_{\{x,y\}\in \cE}\mapsto(\omega(\{x+1,y+1\}))_{\{x,y\}\in \cE}.
	\end{equation*}
	In words we shift all edges by one vertex to the right. Since we endow $\Omega$ with the probability measure $\Pw$ such that  $(\omega(e))_{e\in \cE}$ is a family of independent random variables it is clear that $(\Omega, \cF,\Pw,S)$ is ergodic. Also, it is not difficult to see that $X_k=f(S^{-k}\omega)$ for all $k\in \Z$ for a measurable function $f:\Omega\to\{0,1\}$. Then by Birkhoff's mean ergodic theorem in (\ref{Birkhoff}) it follows that
	\begin{equation*}
		\frac{1}{2n}\sum_{k=-n}^{n}X_k\to \E[X_0]=\Pw(0 \text{ is a cut-point})>0
	\end{equation*}
	almost surely. This implies that infinitely many $X_k$ are equal to $1$ almost surely. The second statement is immediate since there are no edges between disjoint intervals whose boundary points are given by consecutive cut-points, and the edges contained in those intervals are independent. This also means that with probability $1$ there cannot exist an infinitely large component.
\end{proof}
\section{Construction of the CPDLP via a graphical representation}\label{GraphRepCPDLP}
In this section we formally construct the CPDLP via a graphical representation.
First let us define the dynamical percolation. We denote by $\Delta^{\text{op}}=(\Delta^{\text{op}}_e)_{e\in \cE}$ and $\Delta^{\text{cl}}=(\Delta^{\text{cl}}_e)_{e\in \cE}$ two independent families of Poisson point processes on $\R$ such that $\Delta^{\text{op}}_e$ has rate $\hat{v}_{e}\hat{p}_{e}$ and $\Delta^{\text{cl}}_e$ has rate $\hat{v}_{e}(1-\hat{p}_{e})$.
We  define for each edge $e \in  \cE$ a two state Markov process $X_t(e)$ on the state space $\{0,1\}$. Assume that the initial state is $0$, i.e.~$X_0(e)=0$. Set $T_0(e):=0$ and define recursively
\begin{align*}
	T_{2n+1}(e)& :=\inf\{t>T_{2n}: t \in \Delta^{\text{op}}_e\},\\
	T_{2n+2}(e)& :=\inf\{t>T_{2n+1}: t \in \Delta^{\text{cl}}_e\}.
\end{align*}
for any $n\in \N_0$. Then we set $X_t(e)=0$ if $t\in [T_{2n},T_{2n+1})$ and $X_t(e)=1$ if $t\in [T_{2n+1},T_{2n+2})$ for some $n \in \N_0$. If the initial state is $1$, i.e.~$X_0(e)=1$, then we can just interchange the two Poisson point processes in the definition of the stopping times $T_n(e)$. Since the Poisson point processes are independent if $e\neq e'$ we get that $X(e)$ and $X(e')$ are independent as well. 
By construction we see that $X(e)$ has the transitions
\begin{align*}
	0&\to 1 \quad \text{ at rate } \hat{v}_{e}\hat{p}_{e}\text{ and }\\
	1&\to 0	\quad \text{ at rate } \hat{v}_{e}(1-\hat{p}_{e}).
\end{align*}
Note that the stationary distribution of $X(e)$ is a Bernoulli distribution with parameter $\hat{p}_{e}$. 
Now we define $\bfB_t:=\{e\in E: X_t(e)=1\}$, which is a Feller process with state space $\cP(\cE)$
and transitions as in (\ref{DymLongRangePer}).
We add the set $B\subset\cE$ of all initially open edges to indicate the initial state, i.e. $\bfB^{B}_{0}=B$.
In most cases we choose the invariant distribution $\pi$ as initial distribution, i.e.~$\Pw(e\in \bfB_0)=\hat{p}_e$ for any $e$. In this case we omit the superscript.

Next we define the infection process $\bfC$. Let $\Delta^{\inf}=(\Delta^{\inf}_{e})_{e\in \cE}$ and $\Delta^{\text{rec}}=(\Delta^{\text{rec}}_x)_{x\in V}$  be two independent families of Poisson point processes on $\R$ such that for all $e\in \cE, x \in V$ fixed $\Delta^{\inf}_{e}$ has rate $\lambda$ and $\Delta^{\text{rec}}_x$ has rate $r$ and the processes are independent. From here on we intepret $\Delta^{\inf}$ and respectively $\Delta^{\text{rec}}$ as Poisson random sets on $\cE\times \R$ and respectively on $V\times \R$ for the sake of a more transparent notation. We call $(e,t)\in \Delta^{\inf}$ an infection event and $(x,t)\in \Delta^{\text{rec}}$ a recovery event. Next we need to introduce the notion of an infection path.
\begin{definition}
	\label{InfectionPath}
	Let  $(y, s)$ and $(x, t)$ with $s < t$ be two space-time points and  $\bfB$ a background process. We say that there is a \textit{$\bfB$-infection path} from $(y, s)$ to $(x, t)$ if there is a sequence
	of times $s = t_0 < t_1 < \dots < t_n \leq t_{n+1} = t$ and space points $y = x_0,x_1,\dots, x_n = x$ such that
	$(\{x_{k-1},x_{k}\},t_k)\in \Delta^{\inf}$ as well as  $\{x_{k-1},x_{k}\}\in \bfB_{t_k}$ for all $k \in \{ 1, \dots , n\}$ 
	and  $\Delta^{\text{rec}}\cap\big(\{x_k\}\times[t_k , t_{k+1} )\big)=\emptyset$ for all $k \in \{ 0, \dots , n\}$. We write $(y, s)\stackrel{\bfB}{\longrightarrow} (x, t)$ if there exists a $\bfB$-infection path. 
\end{definition}
Now in order to define the infection process $\bfC^{C,B}$ we choose a background process $\bfB^{B}$ with $\bfB^{B}_0=B$ and set $\bfC_0^{C,B}:=C\subset V$ as well as
\begin{align}\label{DefinitionCPDLP}
	\bfC_t^{C,B}:=\{x\in V:\exists y\in C \text{ such that } (y,0)\stackrel{\bfB^{B}}{\longrightarrow} (x,t) \} \quad \text{ for } t > 0.
\end{align}
Again, if the background is stationary we write simply $(\bfC^{C},\bfB)$. See Figure~\ref{fig:GraphicalReperesentation} for an illustration of the graphical representation of the infection process $\bfC$.
\begin{figure}[!b]
	\centering 
	\includegraphics[width=70mm]{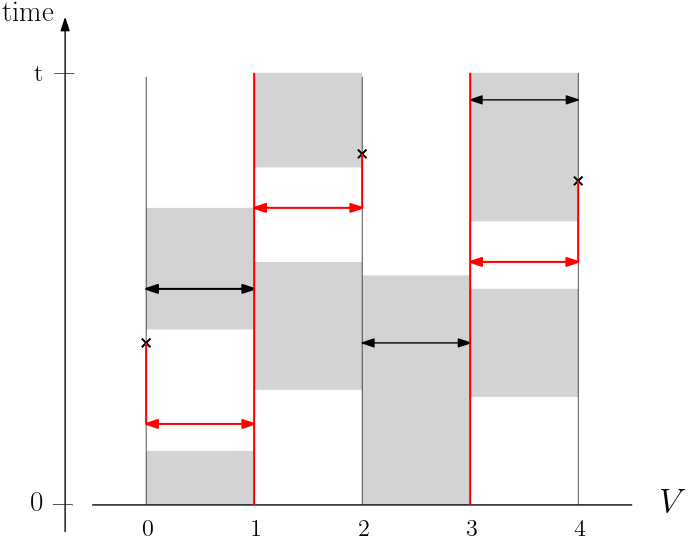}
	\caption{The arrows indicate infection events and crosses indicate recovery events. The grey areas indicate closed edges and red lines indicate infection paths. Note that for the sake of simplicity we illustrated this only for a nearest neighbor structure.}
	\label{fig:GraphicalReperesentation}
\end{figure}
\begin{remark}
	We use infection events which transmit infection from $x$ to $y$ as well as from $y$ to $x$ as indicated by the double headed arrows. In the literature it is also common to use oriented infection arrows instead, which only transmit from $x$ to $y$. It is easy to see that both constructions yield the same dynamics. In general, there is no unique graphical representation, and thus there exists multiple ways to construct the same process.
\end{remark}
\begin{remark}\label{BasicProperties}
	As for the standard contact process the graphical representation  can be used to show that the CPDLP is \emph{monotone increasing} with respect to the initial conditions, i.e.~if $C\subset C'$ and $B\subset B'$ we have for suitably coupled processes $\bfC_t^{C,B}\subset \bfC_t^{C',B'}$ and $\bfB^B_t \subset \bfB_t^{B'}$ for all $t\geq0$. Furthermore, the process is also monotone increasing with respect to the infection rate $\lambda$ and the parameter $q$. Finally, the process is additive with respect to the initially infected vertices, i.e.~$\bfC_t^{C\cup C',B}=\bfC_t^{C,B}\cup \bfC_t^{C',B}$ for all $t\geq 0$ when we are using the same Poisson point processes.
\end{remark}
By Remark~\ref{BasicProperties} it is clear that
for infection kernels of the from (\ref{scaledkernels})
the critical infection rate $\lambda_c(\gamma,q)$
is non-increasing in $q$. This is not clear at all for $\gamma$. But via the graphical representation we can conclude that at least the function $\gamma \mapsto \gamma^{-1}\lambda_c(\gamma,q)$ is non-increasing. This means that the critical infection rate $\lambda_c(\gamma,q)$ can at most increase linearly with respect to $\gamma$, as can be proven analogously to  \cite[Proposition~2.2]{linker2019contact}.
\begin{proposition}\label{WeakSpeedMonotonicity}
	Let $q\in(0,1)$. The function $\gamma \mapsto \gamma^{-1}\lambda_c(\gamma,q)$ is monotone non-increasing.
\end{proposition}

\subsection{The CPDLP is a Feller process}\label{Welldefined&Feller}
By definition it is not clear yet if the CPDLP is a well-defined Feller process. For example it is not clear if it might occur at some time $t$ that a vertex $x$ is connected to infinitely many other vertices via open edges. Since along open edges we consider a constant infection rate $\lambda$ this would lead to an infinite transition rate if $x$ is infected. Furthermore, it is also not clear that if we start with finitely many infected vertices that the set of all infected vertices stays finite for the whole time.

Thus, in this section we show that if we assume \eqref{ExistenceAssumption} the CPDLP is a well-defined Feller process. 
First, we will show that the set of all infections stays finite for all $t>0$ (in expectation and thus almost surely) if we started with a finite number of infected vertices $C$. The same is true for 
the set of vertices that may influence a finite set of vertices at a later time.

To this end we build a CPDLP $(\prescript{}{0}\bfC^{C,B},\bfB^{B})$ via the same graphical representation used to build $(\bfC^{C,B},\bfB^{B})$, just that we only use the infection arrows and ignore all recovery events. This yields that $(\prescript{}{0}\bfC^{C,B},\bfB^{B})$ is a CPDLP with recovery rate $0$ and  $(\bfC^{C,B},\bfB^{B})$ has recovery rate $r$. 
Naturally, we then have $\bfC^{C,B}_t\subset \prescript{}{0}\bfC^{C,B}$	for all $t\geq 0$. We are also interested in the set of vertices in the past that have influenced the state of a particular vertex $x$ at time $t>0.$ Fix an $0<s<t$. For any background process  $\bfB$ we set more generally for $C \in \cP(V)$
\begin{equation}
	\label{dualobject}
	\widehat{\bfC}^{C,t}_s=\widehat{\bfC}^{C,t}_{\bfB,s}
	:=\{y\in V: (y,t-s)\stackrel{\bfB}{\longrightarrow}(x,t) \text{ for some } x \in C\},
\end{equation}
which describes the set of vertices at time $t-s$ relevant for vertices $x \in C$ at time $t$. 
We write $\prescript{}{0}{\widehat{\bfC}}^{C,t}_s$ if for the $\bfB$-infection paths we ignore 
recovery events ($r=0$). Note that this process is increasing in $s$.

Now we will formulate a comparison of $\prescript{}{0}\bfC^{\{x\},\cE}_t$ as well as
$\prescript{}{0}{\widehat{\bfC}}^{C,t}_s$ to a connected component of a long range percolation model. Let us define for every $e\in \cE$ the random variables $X_t(e):=\1_{\{e\in \bfB^{\cE}_t\}}$ and
\begin{align}
	\label{PotentialInfectionEdges}
	Y_t(e):=
	\begin{cases}
		1 & \text{ if } \exists s\leq t \text{ such that } (e,s)  \in \Delta^{\inf} \text{ and } X_s(e)=1,\\
		0 & \text{ otherwise.}
	\end{cases}
\end{align}
Obviously, $(Y_t(e))_{e\in\cE}$ is a family of independent Bernoulli random variables with $b_e:=\Pw(Y_t(e)=1)$.
We declare an edge $e$ to be open if $Y_t(e)=1$ and closed otherwise so that we obtain a long range percolation model. 
We denote for this model the connected component containing $x$ by $\cC^{Y_t}(x)$. Recall from Definition~\ref{InfectionPath} that a $\bfB^{\cE}$-infection path only consists of infection events $(s,e)\in \Delta^{\inf}$ with $e\in \bfB^{\cE}_s$. Thus, it is easy to see that $\prescript{}{0}\bfC^{\{x\},\cE}_t\subset\cC^{Y_t}(x)$
since at least all vertices at time $t$ which are connected to $(0,x)$ via an $\bfB^{\cE}$-infection path are contained in $\cC^{Y_t}(x)$ by the definition of $(Y_t(e))_{e\in \cE}$. This also implies together with monotonicity that 
$\prescript{}{0}\bfC^{\{x\},B}_t\subset\cC^{Y_t}(x)$ for any $B\subset \cE$.
Likewise, since  in \eqref{PotentialInfectionEdges} infection events in both directions are included we have with an analogous argument that
$\prescript{}{0}{\widehat{\bfC}}^{\{x\},t}_{\bfB^{\cE},t}\subset\cC^{Y_t}(x)$.
\begin{lemma}\label{FullEdgeStart}
	For every $\varepsilon>0$ there exists a $T>0$ such that $\E[|\cC^{Y_t}(x)|]<1+\varepsilon$ for all $t\leq T$ and every $x\in V$. This implies in particular that 
	$\E[|\prescript{}{0}\bfC^{C,\cE}_t|]<\infty$
	as well as
	$\E[|\prescript{}{0}{\widehat{\bfC}}^{C,t}_{\bfB^{\cE},t}|]<\infty$
	for all $t\leq T$ if $C\subset V$ is finite.
	We also have 
	$\prescript{}{0}{\widehat{\bfC}}^{\{x\},t}_{\bfB^{\cE},t}
	\subset\cC^{Y_t}(x) \searrow  \{x\}$ almost surely as $t \searrow  0.$
\end{lemma}
\begin{proof}
	Due to translation invariance it suffices to consider an arbitrary $x\in V$ for the first statement of the lemma. 
	Since the rate of potential infections due to $\Delta^{\inf}$ along $e$ is $\lambda$ we see that 
	\begin{equation*}
		b_e=\Pw(Y_t(e)=1)=1-\E\Big[e^{-\lambda\int_{0}^{t}X_s(e)\mathsf{d}s}\Big]\leq \E\Big[\lambda\int_{0}^{t}X_s(e)\mathsf{d}s\Big],
	\end{equation*}
	for every $e\in \cE$, where we used that $1-e^{-x}\leq x$ for all $x\in \R$. The process $X(e)$ is obviously a two-state continuous time Markov chain with jump rates $\hat{v}_e\hat{p}_e$ and $\hat{v}_e(1-\hat{p}_e)$, and thus $\int_{0}^{t}X_s(e)\mathsf{d}s$ is the occupation time of state $1$, i.e. the total time $e$ was open until time $t$. 
	The  first moment can be calculated explicitly, see Pedler \cite{pedler1971occupation}, to give
	\begin{equation*}
		\E\Big[\int_{0}^{t}X_s(e)\mathsf{d}s\Big]=\hat{p}_et+\frac{1-\hat{p}_e}{\hat{v}_e}(1-e^{-\hat{v}_et})\leq \hat{p}_et+\frac{1-e^{-\hat{v}_et}}{\hat{v}_e}.
	\end{equation*}
	By Remark~\ref{Rem:pIsSummable} we know that the right hand side is summable over edges connected to $x$ and each term convergences to $0$ as $t\to 0$. By Lebesgue's dominated convergence theorem the sum also converges to $0$, and thus, for $t$ small enough $\sum_{y\in V\backslash \{x\}} b_{\{x,y\}}<1$. Hence, by Proposition~\ref{lowerBound&integrability} there exists almost surely no infinitely large connected components and moreover $\E[|\cC^{Y_t}(x)|]<\infty$ for every $x\in V$. Furthermore,  $t\mapsto \cC^{Y_t}(x)$ is a.s. monotone decreasing as $t\to 0$ and so $\cC^{Y_t}(x)\searrow \{x\}$. Thus, by monotone convergence the first claim follows.
	
	Since $C$ is finite it suffices by additivity to conclude that
	$\E[|\prescript{}{0}\bfC^{\{x\},\cE}_t|]<\infty$    and
	$\E[|\prescript{}{0}{\widehat{\bfC}}^{\{x\},t}_{\bfB^{\cE},t}|]<\infty$
	for $t \leq T$ for some arbitrary $x\in V$. But since we already know that
	$\prescript{}{0}\bfC^{\{x\},\cE}_t\subset \cC^{Y_t}(x)$ and
	$\prescript{}{0}{\widehat{\bfC}}^{\{x\},t}_{\bfB^{\cE},t}
	\subset \cC^{Y_t}(x)$  for all $t\geq 0$ the second part of the lemma follows immediately from 
	what we have already shown above.
\end{proof}
\begin{proposition}
	\label{propn:finiteinfectedsets2}
	Let $C\subset V$ be non-empty and finite and $B\subset \cE$.  Then
	$\E[|\bfC^{C,B}_t|] \leq \E[|\prescript{}{0}\bfC^{C,B}_t|]<\infty$
	for all $t \geq 0$ as well as
	$\E[|\prescript{}{0}{\widehat{\bfC}}^{C,t}_{\bfB^{B},s}|]<\infty$
	for all $t\geq s \geq 0$.
\end{proposition}
\begin{proof}
	By monotonicity it suffices to show the claim for $B=\cE$.
	Fix $t>0$ and let $T>0$ be as in Lemma \ref{FullEdgeStart}. Then choose $u\leq T$ and $k \in \N$ such that $t=ku.$ 
	We have due to additivity that 
	\begin{equation*}
		\E\big[|\prescript{}{0}{\bfC}^{C,\cE}_u|\big]=\E\big[|\bigcup_{x \in C}\prescript{}{0}{\bfC}^{\{x\},\cE}_u|\big] 
		\leq \sum_{x \in C}\E\big[|\prescript{}{0}{\bfC}^{\{x\},\cE}_u|\big] = |C| \cdot  \E\big[|\prescript{}{0}{\bfC}^{\{x\},\cE}_u|\big]<\infty, 
	\end{equation*}
	where the expectation on the right hand side is independent of $x$ and finite due to $|C|<\infty$ and due to our choice of $u$ according to Lemma \ref{FullEdgeStart}. 
	On subsequent time intervals $[lu,(l+1)u]$ we can use the Markov property and iterate the same argument conditioned on $(\overline{\bfC}^{\{C\},\cE}_{lu},\bfB^{\cE}_{lu})$. Note that in any case $\bfB^{\cE}_{lu} \subset \cE$ and so
	because of monotonicity (with respect to the background) we obtain 
	\begin{equation*}
		\E\big[|\prescript{}{0}{\bfC}^{C,\cE}_{(l+1)u}| \big| (\prescript{}{0}{\bfC}^{C,\cE}_{lu},\bfB^{\cE}_{lu})\big]
		\leq  |\prescript{}{0}{\bfC}^{C,\cE}_{lu}| \cdot  \E[|\prescript{}{0}{\bfC}^{\{x\},\cE}_u|].
	\end{equation*}
	Putting this together we arrive  at 
	\begin{equation*}
		\E\big[|\prescript{}{0}{\bfC}^{C,\cE}_t|\big]\leq  |C|\cdot \E\big[|\prescript{}{0}{\bfC}^{\{x\},\cE}_u|\big]^k<\infty. 
	\end{equation*}
	
	\smallskip
	The argument is completely analogous for 
	$\prescript{}{0}{\widehat{\bfC}}^{C,t}_{\bfB^B,s} \subset \prescript{}{0}{\widehat{\bfC}}^{C,t}_{\bfB^{\cE},s}$.
	In this case we subdivide the time interval $[t-s,t]$ into $k$ subintervals of length $u\leq T$ from Lemma \ref{FullEdgeStart}. We again use that $\bfB^{\cE}_t \subset \cE $ for any $t \geq 0,$ in particular at the start of each subinterval. Due to  additivity in the sense that 
	$$\prescript{}{0}{\widehat{\bfC}}^{C,t}_{\bfB^{\cE},s}= \bigcup_{x \in C}\prescript{}{0}{\widehat{\bfC}}^{\{x\},t}_{\bfB^{\cE},s},$$
	we obtain analogously with the help of Lemma \ref{FullEdgeStart} that 
	\begin{equation*}
		\E\big[|\prescript{}{0}{\widehat{\bfC}}^{C,t}_{\bfB^{\cE},s}|\big]\leq  |C|\cdot \E\big[|\prescript{}{0}{\widehat{\bfC}}^{\{x\},u}_{\bfB^{\cE},u}\big]^k<\infty.\qedhere
	\end{equation*}
\end{proof}

Now we want to show that the CPDLP is a Feller process. We consider the transition kernel of the CPDLP denoted by $P_t\big((C,B),\,\cdot\,\,\big):=\Pw((\bfC_t^{C,B},\bfB_t^{B})\in \,\cdot\,)$ for any $t\geq0$. The transition semigroup is defined by
\begin{equation*}
	T_tf(\cdot)=\int f(y)P_t(\cdot,dy) 
\end{equation*}
for any $t\geq0$. In order to show that $(T_t)_{t\geq 0}$ is a Feller semigroup it suffices to show for any fixed $t\geq 0$ that  $(C,B)\mapsto P_t\big((C,B),\cdot\big)$ is continuous (as a mapping into  $\cM_1(\cP(V)\times \cP(\cE))$ the space of all probability measures on $\cP(V)\times \cP(\cE)$ equipped with the topology of weak convergence). This then implies that $T_t$ maps continuous functions to continuous functions, as well as point-wise convergence, i.e.~$T_tf(C,B)\to f(C,B)$ as $t\to 0$ for every $C\subset V$ and $B\subset \cE$, see \cite[Lemma~17.3]{kallenberg2006foundations}. From this strong continuity already follows, i.e. $\sup_{(C,B)}|T_tf(C,B)-f(C,B)|\to 0$  as $t\to 0$ for continuous $f$, see \cite[Theorem~17.6]{kallenberg2006foundations}. 
\begin{proposition}
	\label{propn:Feller}
	The map $((C,B),t)\mapsto P_t\big((C,B),\,\cdot\,\big)$ is continuous seen as a mapping from $(\cP(V)\times \cP(\cE))\times [0,\infty)$ to $\cM_1(\cP(V)\times \cP(\cE))$, and thus $(T_t)_{t\geq 0}$ is a Feller semigroup.
\end{proposition}
\begin{proof}
	To prove the claim it suffices to show (using the graphical representation) that
	\begin{equation}
		\label{CBconv}
		(\bfC^{C_n,B_n}_{t_n},\bfB^{B_n}_{t_n})\to (\bfC^{C,B}_{t},\bfB^{B}_{t})
	\end{equation}
	almost surely as $((C_n,B_n),t_n)\to ((C,B),t)$. Since we equipped $\cP(V)\times \cP(\cE)$ with the topology which is  induced by pointwise convergences it thus suffices to show that
	\begin{equation*}
		\1_{\{(x,e)\in (\bfC^{C_n,B_n}_{t_n},\bfB^{B_n}_{t_n})\}}\to \1_{\{(x,e)\in (\bfC^{C,B}_{t},\bfB^{B}_{t})\}}
	\end{equation*}
	almost surely as $((C_n,B_n),t_n)\to ((C,B),t)$ for every $(x,e)\in V\times\cE$. Since $\1_{\{e\in \bfB^{B}_{t}\}}$ and $\1_{\{e'\in \bfB^{B}_{t}\}}$ are independent if $e\neq e'$ it is clear that
	\begin{equation}\label{PointwiseContBackground}
		\1_{\{e\in \bfB^{B_n}_{t_n}\}} = \1_{\{e\in \bfB^{B}_{t}\}}
	\end{equation}
	if neither $\Delta^{\text{op}}_e$ nor $\Delta^{\text{cl}}_e$ have points between $t_n$ and $t$ and $\1_{\{e\in B_n\}}=\1_{\{e\in B\}}$. This will 
	almost surely be the case as soon as  $(B_n,t_n)$ is close enough to  $(B,t)$. Thus, it just remains to show convergence of the infection process. Let us first consider $t_n\nearrow t$. 
	Fix an $\varepsilon >0$ and let  $\bfB^{\varepsilon}$ be a background process that is restarted at time $t-\varepsilon$ in the state $\cE$
	and consider $\widehat{\bfC}^{\{x\},\varepsilon}_{\bfB^{\varepsilon},\varepsilon}$.

	Then, Lemma~\ref{FullEdgeStart} yields that 
	$\widehat{\bfC}^{\{x\},\varepsilon}_{\bfB^{\varepsilon},\varepsilon}$
	is finite almost surely and $\widehat{\bfC}^{\{x\},\varepsilon}_{\bfB^{\varepsilon},\varepsilon}
	\searrow \{x\}$ as $\varepsilon\searrow 0$. Since the process is monotone there are fewer infection paths for the actual process because at time $t-\varepsilon$ fewer edges are open. Likewise, as $\varepsilon\searrow 0$ we will have $\Delta^{\text{rec}}_x \cap [t-\varepsilon, t]=\emptyset.$ Thus, almost surely for $n$ large enough we have 
	\begin{equation*}
		\1_{\{x\in \bfC^{C_n,B_n}_{t_n}\}}= \1_{\{x\in \bfC^{C,B}_{t_n}\}} 
		\Rightarrow \1_{\{x\in \bfC^{C_n,B_n}_{t_n}\}}= \1_{\{x\in \bfC^{C,B}_{t}\}}.
	\end{equation*}
	It now remains to show that the condition on the left hand side is fulfilled as $n$ becomes large.
	For this we note from the above that a.s.~for $n$ large enough 
	$\widehat{\bfC}^{\{x\}, t_n}_{ \bfB^{\cE}, t_n}=\widehat{\bfC}^{\{x\}, t}_{ \bfB^{\cE}, t}$.
	We now use that by Proposition~\ref{propn:finiteinfectedsets2}
	$\prescript{}{0}{\widehat{\bfC}}^{\{x\}, t}_{ \bfB^{\cE}, t}$
	is a.s.~a finite set which for all $0\leq s \leq t$ includes 
	$\widehat{\bfC}^{\{x\}, t}_{ \bfB^{B_n}, s}$ for all $n$ and 
	$\widehat{\bfC}^{\{x\}, t}_{ \bfB^{B}, s}$.
	Because we have for $n$ possibly larger still that $C_n$ and $C$ agree on $\prescript{}{0}{\widehat{\bfC}}^{\{x\}, t}_{ \bfB^{\cE}, t}$ the claim now follows.
	Thus, we have shown left continuity. Right continuity follows by an analogous argument. Together with \eqref{PointwiseContBackground} this shows \eqref{CBconv} and so the claimed  continuity.
	As pointed out before the proposition this implies that $(T_t)_{t \geq 0}$ is a Feller semigroup and hence the CPDLP a Feller process.
\end{proof}

\subsection{Phase transition with respect to the non-trivial upper invariant law}\label{InvariantLaw}
In this subsection we study the phase transition with respect to the non-trivial upper invariant law $\overline{\nu}$. To be precise we will show Theorem~\ref{CrticalValuesAgree} which states that the critical infection rate for survival agrees with the critical infection rate for non-triviality of the upper invariant law, i.e. ~$\lambda_c'=\lambda_c$.

In the last section we proved that the CPDLP $(\bfC,\bfB)$ is a well-defined Feller process on the state space $\cP(V)\times \cP(\cE)$ with  corresponding Feller semigroup $(T_t)_{t\geq0}$. We denote by $\nu T_t$ the distribution of the CPDLP at time $t\geq 0$ when the initial distribution is $\nu$.

Recall that the upper invariant law $\overline{\nu}$ is the weak limit of  $\big((\delta_{V}\otimes \delta_{\cE})T_t\big)_{t\geq 0}$ as $t\to \infty$, and that it dominates every other invariant law $\nu$ of the CPLDP with respect to the stochastic order.  
In fact, even if we start the background process stationary, i.e. $\bfB_0\sim\pi$, instead of in $\cE$ the CPLDP still convergences towards the upper invariant law $\overline{\nu}$. Analogously to \cite[Lemma~6.2]{seiler2022contact} one can show that if $\mu$ is a probability measure such that $\pi\preceq \mu$ then $(\delta_{V}\otimes\mu)T_t\Rightarrow \overline{\nu}$  as $t\to \infty$.

Recall the definition of the process $\widehat{\bfC}$ from \eqref{dualobject}, i.e.~for any fixed  $t\geq 0$ and $C\subset V$ we have
\begin{equation*}
	\widehat{\bfC}^{C,t}_{\bfB,s}
	=\{y\in V: (y,t-s)\stackrel{\bfB}{\longrightarrow}(x,t) \text{ for some } x \in C\}.
\end{equation*}
One way of interpreting this definition is the following. First we fix a time $t$ and a set of infected vertices $C$. Next we fix the background $\bfB$ and let the graphical representation run backwards in time from $t$ to $t-s$. By this train of thought it is easy to see that the following \emph{duality} relation holds. Let $A,C\subset V$ and $B\subset \cE$, then 
\begin{align}\label{ConditinalDuality}
	\Pw\big(\bfC^{C,B}_{t}\cap A \neq \emptyset \big|\bfB\big)=\Pw\big(\bfC^{C,B}_{t-s}\cap \widehat{\bfC}^{A,t}_{\bfB^{B},s}\neq \emptyset \big|\bfB\big)=\Pw\big(C\cap \widehat{\bfC}^{A,t}_{\bfB^{B},t} \neq \emptyset \big|\bfB\big) 
\end{align}
holds almost surely for all $s\leq t$. For a contact process without background this procedure yields a so called self-duality relation. For the infection process $\bfC$ this is not always the case since if the background is started in an arbitrary initial condition $B\subset \cE$, then $(\widehat{\bfC}^{t},\widehat{\bfB}^{t})$ will in general not be a CPDLP. Only if we start the background process stationary, i.e.~$\bfB_0\sim\pi$, can we recovery the self-duality. The main reasons for this is that $\bfB$ is reversible with respect to $\pi$, i.e.~if $\bfB_0\sim\pi$, then $(\bfB_s)_{s\in[0,t]}\stackrel{d}{=} (\bfB_{t-s})_{s\in[0,t]}$. Thus, we get that 
\begin{equation*}
	\Pw(\bfC^C_{t}\cap A \neq \emptyset)=\Pw(\bfC^A_{t}\cap C \neq \emptyset).
\end{equation*}
For a detailed proof of this equality see \cite[Proposition~6.1]{seiler2022contact}. Note that as already mentioned in the beginning of Section~\ref{GraphRepCPDLP} we dropped the super and subscripts concerning the background process since we consider it to be stationary. 

This self duality relation enables us to deduce the following connection between the survival probability $\theta$ and the particle density of the upper invariant law $\overline{\nu}$.
\begin{equation*}
	\theta(C)=\Pw(\bfC^C_t\neq \emptyset\,\, \forall t\geq 0)=\overline{\nu}(\{A\subset \cP(V):C\cap A\neq\emptyset\}\times \cP(\cE)).
\end{equation*}
which can be used to show that the two critical infection rates agree, i.e.~$\lambda_c'=\lambda_c$, and therefore show Theorem~\ref{CrticalValuesAgree}. For a detailed proof see \cite[Proposition~6.3]{seiler2022contact}.
\begin{remark}
	While the results in \cite[Section~6]{seiler2022contact} that we refer to in this subsection are stated there for graphs with uniformly bounded degrees, which is not the case for the CPDLP,  the proofs of these specific results do not rely on this property and can thus be applied in the exact same way in this setting. 
\end{remark}

\subsection{Comparison with a contact process}\label{ConsequencesOfTheGraphRep}
In this subsection we prove Theorem~\ref{ComparisonWithLongRangeCP}, which provides a comparison between a contact process with a specific infection kernel and the CPDLP, and Corollary~\ref{CritRateFastSpeed}. We will see that this contact process acts as a lower bound with respect to survival, i.e if the contact process survives so does the CPDLP.

Recall the contact process with general infection kernel from 
\eqref{GeneralInfection}.
Now we show Theorem~\ref{ComparisonWithLongRangeCP}, which states that we can couple the CPDLP $(\bfC^{C},\bfB)$ with a contact process $\overline{\bfX}^C$ with infection kernel 
$(\overline{a}_{e}(\lambda))_{e\in \cE}$, where
\begin{align*}
	\overline{a}_{e}(\lambda)=\frac{1}{2}\Big(\lambda+ \hat{v}_{e}-\sqrt{(\lambda+\hat{v}_{e})^2- 4\lambda \hat{v}_{e}\hat{p}_{e}}\,\Big)
\end{align*}
and the same recovery rate such that $\overline{\bfX}^{C}_t\subset \bfC^{C}_t$ for all $t\geq 0$.
\begin{proof}[Proof of Theorem~\ref{ComparisonWithLongRangeCP}]
	Recall the independent two state Markov processes $(X(e))_{e \in \cE}$, where $X_t(e)=\1_{\{e\in \bfB_{t}\}}$ for all $t\geq 0$. 
	The process $X(e)$    has transitions
	\begin{align*}
		0&\to 1	\quad \text{ at rate } \hat{v}_e\hat{p}_e \text{ and }\\
		1&\to 0	\quad \text{ at rate } \hat{v}_e(1-\hat{p}_e).
	\end{align*}
	We now set for all $e\in \cE$
	\begin{align*}
		Y_t(e):=&|\{s\leq t:(e,s)\in\Delta^{\text{inf}} \text{ and }X_s(e)=1\}|.
	\end{align*}
	Note that the intensity measure of the process $Y(e)$ depends on $X(e)$. The process $Y(e)$ is sometimes called a doubly stochastic Poisson process. Given that $X(e)$ is currently in state $x\in\{0,1\}$ the transitions of $Y(e)$ which is currently in state $n$ are
	\begin{equation*}
		n\to n+1	\quad \text{ at rate } \lambda x.
	\end{equation*}
	Now \cite[Theorem~1.4]{broman2007stochastic} together with Strassen's theorem yields that there exist independent Poisson processes $(\overline{Y}(e))_{e \in \cE}$
	with rate 
	\begin{align*}
		\overline{a}_{e}=\frac{1}{2}\Big(\lambda+ \hat{v}_{e}-\sqrt{(\lambda+\hat{v}_{e})^2- 4\lambda \hat{v}_{e}\hat{p}_{e}}\,\Big)
	\end{align*}
	such that $Y_t(e)\geq \overline{Y}_t(e)$ almost surely for all $t\geq 0$ and $e \in \cE.$
	This means that we can find a family of independent Poisson point processes $(\overline{\Delta}^{\inf}_e)_{e \in \cE}$ 
	such that $\overline{\Delta}^{\inf}_e$ has rate $\overline{a}_e$ 
	for $e\in \cE$ and such that $(e,t)\in \overline{\Delta}^{\inf}$ implies that
	$(e,t)\in \Delta^{\inf}$ and $e\in \bfB_{t}$.
	
	Thus, we can construct a contact process $\overline{\bfX}$ on $\cP(V)$ via the graphical representation, with respect to the Poisson point process $\overline{\Delta}^{\inf}$ such that it has the required transition rates and $\overline{\bfX}^{C}_t\subset \bfC^{C}_t$ for all $t\geq 0$. It only remains to show that $\overline{\bfX}$ is a well-defined Feller process. To show this it suffices to verify \eqref{RateBound}. We see that
	\begin{equation}\label{HelpRateBound}
		0 \leq \overline{a}_{e}=\frac{\lambda+\hat{v}_{e}}{2}\Big(1-\sqrt{1- \tfrac{4\hat{v}_{e}\hat{p}_{e}\lambda}{(\lambda+\hat{v}_{e})^2}}\,\Big)\leq \frac{2\lambda \hat{v}_{e}\hat{p}_{e}}{\lambda+\hat{v}_{e}},
	\end{equation}
	where we used that $1-x\leq \sqrt{1-x}$ for $x\in [0,1]$ as well as $\tfrac{4\hat{v}_{e}\hat{p}_{e}\lambda}{(\lambda+\hat{v}_{e})^2}\in[0,1]$ which follows from the fact that
	$(\lambda+\hat{v}_{e})^2\geq4\hat{v}_{e}\lambda$ is equivalent to $(\lambda-\hat{v}_{e})^2\geq 0$ and $\hat{p}_e \in [0,1]$.  Since $\frac{ \hat{v}_{e}}{\lambda+\hat{v}_{e}}\leq 1$ we see that $	\overline{a}_{e}\leq 2\lambda\hat{p}_{e}$. But by \eqref{ExistenceAssumption} the sequence $(\hat{p}_{\{x,y\}})_{y\in V}$ is summable for every $x\in V$, and thus \eqref{RateBound} is satisfied.
	The last claim of the theorem follows immediately.
\end{proof}
Next, we work towards the proof of Corollary~\ref{CritRateFastSpeed}. We thus assume 
kernels of the form given in (\ref{scaledkernels}). As a first result we 
show that the rates $(\overline{a}_{e})_{e\in \cE}=(\overline{a}_{e}(\gamma))_{e\in \cE}$ 
viewed as functions of $\gamma$ converge to $(\lambda \hat{p}_e)_{e\in \cE}$ as $\gamma \to \infty$.
\begin{lemma}\label{AsympFastSpeedLimitRates}
	Let the sequence $(\overline{a}_{e}(\gamma))_{e\in \cE}$ be chosen as in Theorem~\ref{ComparisonWithLongRangeCP}. Then, it follows that for every $x\in V$ 
	\begin{equation}
		\label{a_econv}
		\lim_{\gamma\to \infty} \sum_{y\in V\backslash\{x\}} |\overline{a}_{\{x,y\}}(\gamma)-\lambda \hat{p}_{\{x,y\}}|\to 0.
	\end{equation}
	as well as $\overline{a}_{e}(\gamma)\nearrow\lambda \hat{p}_e$ for all $e\in \cE$ as $\gamma \to \infty$.
\end{lemma}
\begin{proof}
	Let us consider the function $x\mapsto\sqrt{1-x}$ for $x \in [0,1]$. The Taylor expansion at $x=0$ yields that
	\begin{equation*}
		\sqrt{1-x}=1-\frac{x}{2}-O(x^2).
	\end{equation*}
	Thus, since $\tfrac{4\hat{v}_{e}\hat{p}_{e}\lambda}{(\lambda+\hat{v}_{e})^2}\in[0,1]$ 	we obtain
	\begin{align*}
		1-\sqrt{1-\frac{4\gamma v_{e}\hat{p}_{e}\lambda}{(\lambda+\gamma v_{e})^2}}= \frac{1}{2}\frac{4\gamma v_{e}\hat{p}_{e}\lambda}{(\lambda+\gamma v_{e})^2} +O(\gamma^{-2}),
	\end{align*}
	where $O(\gamma^{-2})$ is meant with respect to $\gamma \to \infty$. This implies that
	\begin{align*}
		\overline{a}_{e}(\gamma)=\frac{\lambda+\gamma v_{e}}{2}\Big(1-\sqrt{1- \frac{4\gamma v_{e}\hat{p}_{e}\lambda}{(\lambda+\gamma v_{e})^2}}\,\Big)=\frac{\gamma v_{e}\hat{p}_{e}\lambda}{\lambda+\gamma v_{e}}+O(\gamma^{-1}).
	\end{align*}
	The remainder vanishes and $\tfrac{\gamma v_{e}}{\lambda+\gamma v_{e}}\to 1$ as $\gamma\to \infty$ so that  $\overline{a}_{e}(\gamma)\to \lambda \hat{p}_{e}$. 
	Next let us consider the derivative of $\overline{a}_e(\gamma)$ with respect to $\gamma$, which is 
	\begin{equation*}
		\partial_\gamma\overline{a}_e(\gamma)=\frac{1}{2}\Big( v_{e}-\frac{v_e(\lambda+\gamma v_{e})- 2 v_{e}\hat{p}_{e}\lambda }{\sqrt{(\lambda+\gamma v_{e})^2- 4v_{e}\hat{p}_{e}\lambda\gamma }}\,\Big).
	\end{equation*}
	One can directly calculate that the fraction on the right hand side is always smaller than $v_e$, and therefore $\partial_\gamma\overline{a}_e(\gamma)>0$ for all $\gamma>0$. This implies that $\overline{a}_e(\gamma)$ is monotone increasing for all $e\in \cE$, and thus $0\leq \overline{a}_{e}(\gamma)\nearrow \lambda \hat{p}_{e}$ as $\gamma\to\infty$. 
	This implies
	\begin{equation*}
		\Big|\overline{a}_e(\gamma)-\lambda\hat{p}_e\Big|= \lambda\hat{p}_e-\overline{a}_e(\gamma)
		\leq \lambda\hat{p}_e
	\end{equation*}
	for every $e\in \cE$, which yields that $(\overline{a}_{\{x,y\}}(\gamma)-\lambda \hat{p}_{\{x,y\}})_{y\in V\backslash\{x\}}$ is summable for every $x\in V$. Together with the pointwise converge we proved above we get \eqref{a_econv}.
\end{proof}
We have shown that the rates $(\overline{a}_e)_{e\in\cE}$ converge to the sequence $(\lambda \hat{p}_e)_{e\in \cE}$  from below. Heuristically speaking this justifies to some extend the believe that the infection $\bfC$ converges to the process $\hat{\bfX}$ in the sense of Conjecture~\ref{LimitFastSpeed}, i.e.~that the respective critical values converge. We were not able to show this claim, but we show now that $\lambda_c^{\infty}$ at least acts as an upper bound, i.e.~$\limsup_{\gamma\to \infty}\lambda_c(\gamma)\leq \lambda_c^{\infty}$.

\begin{proof}[Proof of Corollary~\ref{CritRateFastSpeed}] 
	We fix  $q$ and omit it as an index throughout the proof.
	We first observe that by Theorem~\ref{ComparisonWithLongRangeCP} 
	$$\lambda':=\limsup_{\gamma\to \infty}\overline{\lambda}_c(\gamma)
	\geq \limsup_{\gamma\to \infty}\lambda_c(\gamma).$$
	Thus, in order to prove the inequality in the  claim of the corollary it suffices to show $\lambda' \leq \lambda_c^{\infty}.$ 
	That $\lambda'= \lim_{\gamma\to \infty}\overline{\lambda}_c(\gamma)= \lambda_c^{\infty}$ follows then also  
	because by Lemma~\ref{AsympFastSpeedLimitRates} we know that $\overline{a}_e(\lambda,\gamma)\nearrow \lambda  \hat{p}_e$ for every $e\in\cE$, and hence we get  $\lambda_c^{\infty}\leq \liminf_{\gamma\to \infty}\overline{\lambda}_c(\gamma).$ In order to show $\lambda'\leq \lambda_c^{\infty}$
	we first need another bound on the rates $\overline{a}_e=\overline{a}_e(\lambda,\gamma)$.
	
	We know that $\sqrt{1-x}\leq 1-\frac{x}{2}$ for $0\leq x\leq 1$, and thus analogously to
	\eqref{HelpRateBound} we see that
	\begin{equation}
		\label{a_eLowerBound}
		\overline{a}_e(\lambda,\gamma)=\frac{\lambda+\hat{v}_{e}}{2}\Big(1-\sqrt{1- \tfrac{4\hat{v}_{e}\hat{p}_{e}\lambda}{(\lambda+\hat{v}_{e})^2}}\,\Big)\geq \Big(\frac{\hat{v}_{e}}{\lambda+\hat{v}_{e}}\Big)\lambda \hat{p}_{e}.
	\end{equation}
	We know that $\frac{\gamma v_{e}}{\lambda+\gamma v_{e}}\nearrow 1$ as $\gamma v_{e}\nearrow \infty$ for all $e\in \cE$ and since $(v_{e})_{e \in \cE}$ is uniformly bounded below  due to \eqref{ExistenceAssumption} there exists for every $\varepsilon>0$ a $\gamma'$ large enough  such that $\frac{\gamma v_{e}}{\lambda+\gamma v_{e}}>1-\varepsilon$ for all $e\in \cE$ and all $\gamma>\gamma'$. Thus, we get from \eqref{a_eLowerBound} for all $e\in \cE$ that  $\overline{a}_e(\lambda,\gamma)\geq (1-\varepsilon)\lambda \hat{p}_e$
	for all $\gamma>\gamma'$, and so 
	$\lambda' \leq \frac{1}{1-\varepsilon}\lambda_c^{\infty}.$ By letting $\varepsilon$ tend to zero 
	$\lambda' \leq \lambda_c^{\infty}$ follows.
	
	Finally, since we are assuming spatial translation invariance for the kernel $(\hat{p}_e)_{e \in \cE}$ as well as $\hat{p}_{e'}>0$ for at least one $e' \in \cE$ we will have survival of our contact process $\bfX^{\infty}$ if the classical contact process on $\Z$ with infection parameter $\lambda  \hat{p}_{e'}$ survives. But for this contact process the critical infection rate for survival is known to be finite which implies $\lambda_c^{\infty}<\infty.$
\end{proof}
\section{Comparison of the background with a long range percolation model}
\label{BlockwiseLongRangePercolation}
The aim of this section is to prove Theorem~\ref{ImmunThm} and Theorem~\ref{AsymptoticSlowSpeedThm}. These proofs can be found respectively in Subsection~\ref{ImmunizationRegime} and Subsection~\ref{SlowSpeedRegime}. In order to prove these results we need a another infection process $\widetilde{\bfC}$, which dominates the original infection process $\bfC$, i.e.~$\bfC_t\subset\widetilde{\bfC}_t$ for all $t\geq0$, and which
is somewhat easier to analyze.

The idea, which was already used by \cite{linker2019contact} for graphs with bounded degree,  is to compare the dynamical long range percolation $\bfB$ blockwise to a long range percolation model. We partition the time axis $[0,\infty)$ into equidistant intervals $[nT,(n+1)T)$, where $T>0$ and $n\in \N_0$. We also define for each edge $e\in \cE$,
\begin{align}\label{ClosedEdge}
	w_{n}(e)
	:=\begin{cases}
		1 &  \text{if } e\notin \bfB_t \text{ for all } t\in [nT,(n+1)T),\\
		0& \text{otherwise, } 
	\end{cases}
\end{align}
which indicates whether an edge $e$ is closed for the whole time period $[nT,(n+1)T)$. To simplify notation we write $w_n(x,y)$ instead of $w_n(\{x,y\})$ for $\{x,y\}\in \cE$. Now if we accept all infection events $(e,t)\in \Delta^{\inf}$ with $t\in [nT,(n+1)T)$ such that $w_{n}(e)=0$, this leads to an infection process, which survives more easily than $\bfC$, see also the illustration of the graphical representation in \autoref{fig:Reduction}, where we have for the sake of simplicity only drawn  the process with nearest neighbor interactions. 
\begin{figure}[ht]
	\centering 
	\subfigure{\label{fig:Reduction:a}\includegraphics[width=70mm]{GraphicalRepresentationRanInf2}}\hfill
	\subfigure{\label{fig:Reduction:b}\includegraphics[width=70mm]{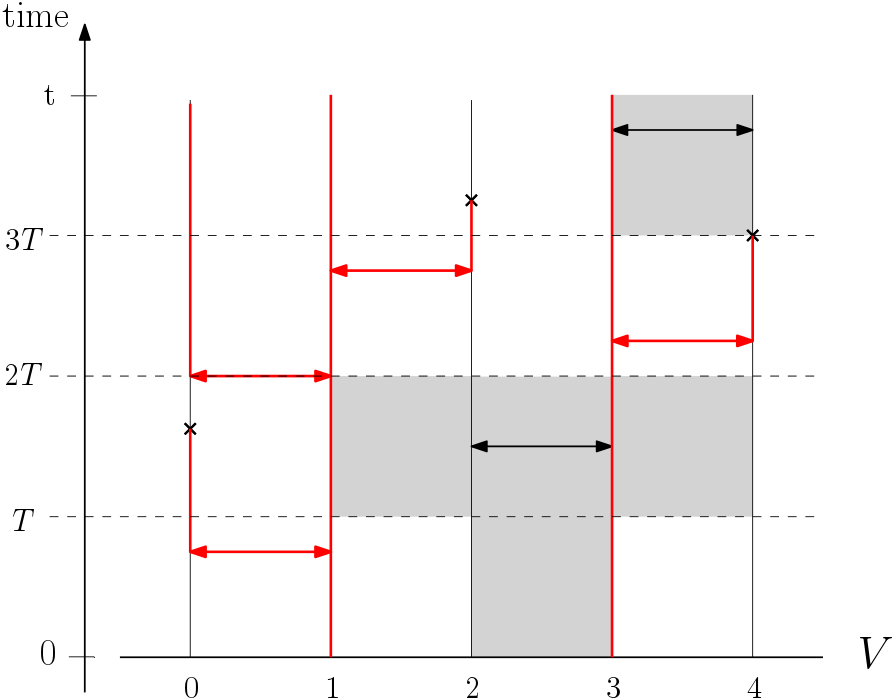}}
	\caption{ On the left hand side we illustrate the graphical representation with respect to the background $\bfB$. On the right hand side we modified the background in such a way that the edges are only closed if they are closed throughout a whole interval $[nT,(n+1)T)$. In both pictures red lines indicate infection paths.}
	\label{fig:Reduction}
\end{figure}
A problem is that obviously $(w_n(e))_{(n,e)\in \N_0\times \cE}$ is not a family of independent random variables. But at least we know that $w_n(e)$ and $w_m(e')$ are independent  for all $n,m\in \N_0$ as long as $e\neq e'$. In order to deal with the dependence that occurs along the time line for a fixed edge
$e$ we need a lower bound on the conditional probability that $w_n(e)=1$ given all previous states $w_{n-1}(e), \dots, w_{0}(e)$, which was shown in \cite[Proposition~3.9]{linker2019contact}.
\begin{proposition}\label{EdgeBound}
	Let $T>0$ be fixed. Then we have for all $n\in \N$ and $e\in E$ that 
	\begin{equation}\label{delta_edef}
		\begin{aligned}
			&\Pw(w_n(e)=1|w_{n-1}(e),\dots, w_{0}(e))\\
			&\geq (1-\hat{p}_{e})e^{-\hat{p}_{e}\hat{v}_{e}T}\frac{e^{-\hat{v}_{e}T}+(1-\hat{p}_{e})(1-e^{-\hat{v}_{e}T})-e^{-\hat{p}_{e}\hat{v}_{e}T}}{1-e^{-\hat{p}_{e}\hat{v}_{e}T}}\\
			&=(1-\hat{p}_{e})e^{-\hat{p}_{e}\hat{v}_{e}T}\Big(1-\hat{p}_{e}\frac{1-e^{-\hat{v}_{e}T}}{1-e^{-\hat{p}_{e}\hat{v}_{e}T}}\Big)=:\delta_{e}(\gamma,q,T)=\delta_{e}.
		\end{aligned}
	\end{equation}
\end{proposition}
The next lemma allows us to compare a family of dependent Bernoulli random variables with an independent family provided that  we have a lower bound on the conditional probabilities.
\begin{lemma}\label{BerComp}
	Let $(X_{n})_{n\in \N_0}$ be a family of Bernoulli random variables such that 
	\begin{align*}
		\Pw(X_{n}=1|X_{n-1},\dots,X_{0})\geq q \in(0,1), \quad n \in \N_0.
	\end{align*}
	Then there exists 
	a family of i.i.d\@. Bernoulli random variables $(X'_{n})_{n\in\N_0}$ such that $\Pw(X_n'=1)=q$ and $X_{n}\geq X'_{n}$ almost surely for every $n\in\N_0$.
\end{lemma}

\begin{proof}
	First of all we set $p_0:=\Pw(X_0=1)$ and for $n\geq 1$ and  $x_{n-1},\dots, x_0\in \{0,1\}$,
	\begin{equation*}
		p_n(x_{n-1},\dots,x_0):=\Pw(X_n=1|X_{n-1}=x_{n-1},\dots,X_{0}=x_0),
	\end{equation*}
	which are by assumption all bounded below by $q$.
	Let $(\chi_n)_{n\in \N_0}$ be an i.i.d.\ family of  random variables, which are uniformly distributed on  $[0,1]$ and also independent of the family $(X_n)_{n\in \N_0}$.
	
	Next we iteratively define the desired family of random variables $(X_n')_{n\in \N_0}$
	along with a family of auxiliary random variables $(Y_n)_{n\in \N_0}$. First, let $Y_0:=\1_{\{\chi_0\leq q_0\}}$ for $q_0:= \frac{q}{p_0} \in [0,1]$. 
	Now set $X_0':=X_0Y_0$ such that  that $X_0'\leq X_0$ and 
	\begin{equation*}
		\Pw(X_0'=1)=\Pw(X_0=1)\Pw(Y_0=1)=p_0q_0=q.
	\end{equation*}
	Next suppose that we already defined $X'_{n-1},\dots,X'_{0}$ as a function of $X_{n-1},\dots,X_{0}$ and $\chi_{n-1}, \dots, \chi_{0}$. We set $p_0':=p_0$ and for $n \geq 1$
	and $x_{n-1},\dots, x_0\in \{0,1\}$,
	\begin{align*}
		p'_n(x_{n-1},\dots,x_0):=\Pw(X_n=1|X'_{n-1}=x_{n-1},\dots,X'_{0}=x_0) \geq q
	\end{align*}
	by our assumption on $(X_{n})_{n\in \N_0}$ and the construction of $X'_{n-1},\dots,X'_{0}$ as a function of $X_{n-1},\dots,X_{0}$ and another independent input from $(\chi_{n})_{n\in \N_0}$. This implies that 
	$$q_n(x_{n-1},\dots,x_{0}):=q\cdot\big(p'_n(x_{n-1},\dots, x_0)\big)^{-1} \in [0,1].$$
	We now set $Y_n:=\1_{\{\chi_n\leq q_n(X'_{n-1},\dots,X'_{0})\}}$ and $X_n':=X_nY_n$. It is again immediately clear that $X_n'\leq X_n$. Also, 
	\begin{align*}
		\Pw(X_n'=1|X'_{n-1},\dots,X'_{0})=\Pw(Y_n=1,X_n=1|X'_{n-1},\dots,X'_{0}).
	\end{align*}
	By choice $\chi_n$ is independent of $(X'_k)_{k\leq n-1}$ and $(X_k)_{k\leq n}$. The random variable $Y_n$ is a function of $\chi_n$ and all $(X'_k)_{k\leq n-1}$. 
	This yields that $Y_n$ and $X_n$ are conditionally independent given $(X'_k)_{k\leq n-1}$ , i.e.
	\begin{align*}
		\Pw(X_n'=1|X'_{n-1},\dots,X'_{0})=&\Pw(X_n=1|X'_{n-1},\dots,X'_{0})\Pw(Y_n=1|X'_{n-1},\dots,X'_{0})\\
		=& p'_n(X'_{n-1},\dots, X'_0) q_n(X'_{n-1},\dots,X'_{0})=q
	\end{align*}
	due to our choice of $q_n.$
	Since the right hand side is independent of the values of $X'_0,\dots,X_{n-1}'$ it follows that $X_n'$ is independent of $(X'_k)_{k\leq n-1}$ and that $\Pw(X_n'=1)=q$. 
	This concludes the proof.
\end{proof}
As a direct consequence of the bounds derived in Proposition~\ref{EdgeBound} together with Lemma~\ref{BerComp} as well as the independence of $w_n(e)$ and $w_m(e')$ for all $n,m\in \N$ as long as $e\neq e'$ we obtain the following  comparison of $(w_n(e))_{n\geq0}$ with a family of i.i.d.\ Bernoulli random variables.
\begin{corollary}\label{ComparisonLongRangePerc}
	Let $T>0$ and let $(w_{n}(e))_{(n,e)\in \N_0\times \cE}$ be defined as in \eqref{ClosedEdge} as well as $\delta_e$ as in \eqref{delta_edef}. Then there exists a family of independent Bernoulli variables  $(w'_{n}(e))_{(n,e)\in \N_0\times \cE}$ such that $\Pw(w'_{n}(e)=1)=\delta_e$ and $w_{n}(e)\geq w'_{n}(e)$ almost surely for all $(n,e)\in \N_0\times \cE$. 
\end{corollary}

Finally we are able to define the new infection process $\widetilde{\bfC}$. We do this analogously to the original process $\bfC$, just that we use infection events $(e,t)\in \Delta^{\inf}$ whenever $w'_{\lfloor t/T\rfloor}(e)=0$. This corresponds to the definition of an infection process as in \eqref{InfectionRatesWithBackground} with the background process $\bfB'$ given by 
\begin{equation}
	\label{connectingpath}    
	\bfB'_{t}= 1-w'_{\lfloor t/T\rfloor}(e), t \geq 0.
\end{equation}
We call a $\bfB'$-infection path as in Definition \ref{InfectionPath} a \emph{connecting path}.

Recall that in the definition of $\bfC$ we only consider infection events $(e,t)\in \Delta^{\inf}$ whenever $e\in\bfB_t$. But by definition this implies that $w_{\lfloor t/T\rfloor}(e)=0$, and by Corollary~\ref{ComparisonLongRangePerc} also  $w'_{\lfloor t/T\rfloor}(e)=0$. Hence, we only get more infection events for $\widetilde{\bfC}$ and thus $\bfC^C_t \subset \widetilde{\bfC}^C_t$ for all $t\geq 0$.

\subsection{Existence of an immunization phase}\label{ImmunizationRegime}
In this subsection we prove Theorem~\ref{ImmunThm} 
which states that for given speed parameter $\gamma>0$ there exists a $q_0\in (0,1)$ such that $\bfC$ dies out almost surely for all $q <q_0$ regardless of the choice of $\lambda>0$, i.e.~$\lambda_c(\gamma,q)=\infty$ for all $q<q_0$, and that $\lambda_c(\gamma,q)<\infty$ for all $q>q_0$.

The idea is that, if $q$ is small enough, then an arbitrary vertex will eventually be isolated for a long time, and therefore a potential infection cannot spread to another vertex before the isolated vertex is affected by a recovery event. 
To make this precise we recall $(w'_{n}(e))_{(n,e)\in \N_0\times \cE}$ from Corollary \ref{ComparisonLongRangePerc} and define $X=(X_{e,n})_{(e,n)\in \cE\times \N_0}$ where $X_{e,n}:=1-	w'_{n}(e)$, as well as  $U=(U_{x,n})_{(x,n)\in V\times \N_0}$ by
\begin{align}
	\label{Udef}
	U_{x,n}&:=
	\begin{cases}
		1 &  \text{if } \Delta^{\text{rec}}\cap (\{x\}\times [nT,(n+1)T))=\emptyset,\\
		0& \text{otherwise. } 
	\end{cases}
\end{align}
If $U_{x,n}=0$ and $\sum_{y\in V\backslash\{x\}}X_{\{x,y\},n}=0$, then an infection on vertex $x$ cannot possibly survive in the time interval $[nT,(n+1)T)$, for any $\lambda>0$. This follows since $\sum_{y\in V}X_{\{x,y\},n}=0$ implies that for the whole time interval all edges attached to $x$ are closed. Therefore, since $U_{x,n}=0$ we know that the vertex $x$ will recover and cannot be reinfected. Furthermore, between time $nT$ and $(n+1)T$ no infection can spread from $x$. Now we define a random graph $G_1$ with vertex set $V\times \N_0$ and add edges according to the following rules.
\begin{enumerate}
	\item If $U_{x,n}=1$, add an oriented edge from $(x,n)$ to $(x,n+1)$.
	\item If $X_{e,n}=1$ for $e=\{x,y\}$, add edges as if $U_{x,n}=1$, $U_{y,n}=1$ and add an unoriented edge between $(x,n)$ and $(y,n)$.
\end{enumerate}
The rules are visualized in \autoref{fig:IsolationGraph1}. Note that all ``horizontal" edges are unoriented such that they can be used in both directions, but all ``vertical" edges are oriented and only point upwards. 
\begin{figure}[t]
	\centering 
	\includegraphics[width=110mm]{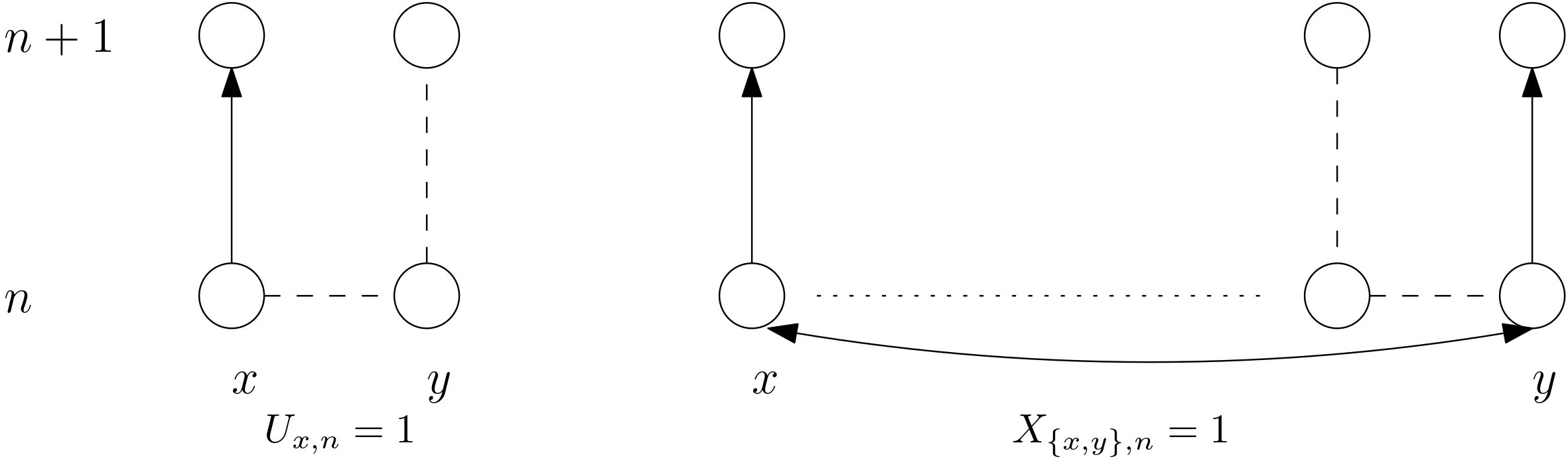}
	\caption{Illustration of the first and second rule. Solid lines indicate present edges and dashed lines absent edges. 
	}
	\label{fig:IsolationGraph1}
\end{figure}

\begin{definition}
	\label{validpathY}
	Let $G_1$ be the random graph constructed above and $C\subset V$ be the set of all initially infected individuals. We say that there exists a \emph{valid path} from  $C\times\{0\}$ to a point $(x,n)$ if there exists a sequence $x_0, x_1,\dots, x_m=x$ with $x_0\in C$ and $0=n_0\leq n_1\leq \dots\leq n_m=n$ such that there exist edges in $G_1$ from $(x_k,n_k)$ to $(x_{k+1},n_{k+1})$ for all $k\in \{0,\dots,m-1\}$.
	
	For every $n\in \N$ we denote by $Y_n=Y_n(U,X)$ the set of all vertices $x\in V$ such that there exists a valid path from $Y_0\times\{0\}$ to $(x,n)$.
\end{definition}
Note that a valid path travels along edges in the direction of their orientation, respectively in both directions in the case of unoriented edges.  In \autoref{fig:IsolationGraph2} we visualize a part of the graph $G_1$ with a valid path. 
\begin{figure}[!b]
	\centering 
	\includegraphics[width=90mm]{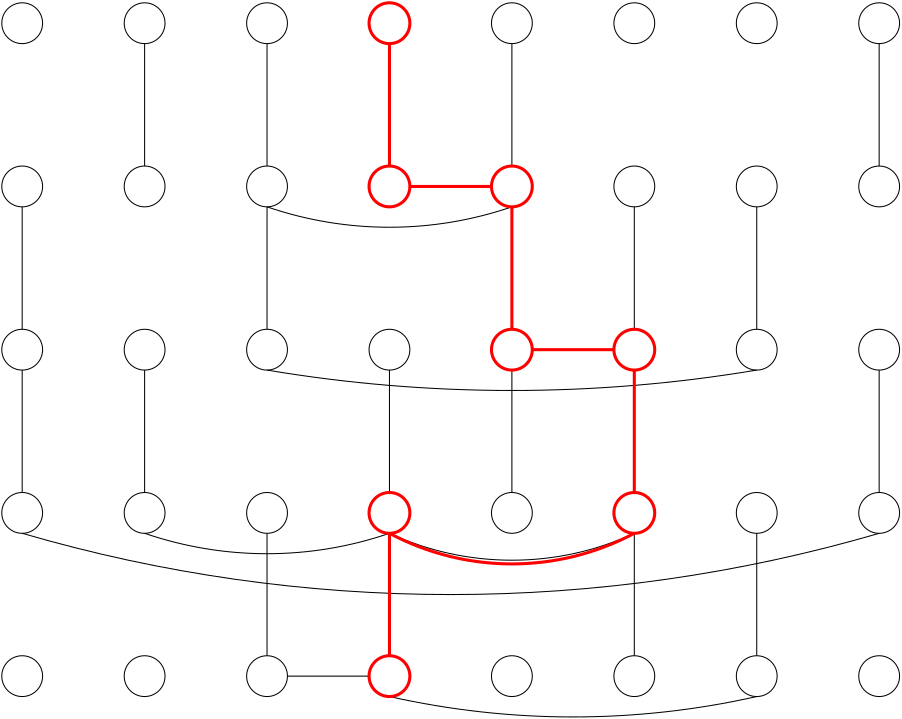}
	\caption{Illustration of a part of $G_1$. The red lines indicate a valid paths.}
	\label{fig:IsolationGraph2}
\end{figure}

\begin{lemma}\label{ExtinctionLemma}
	Let $T>0$, $n\in\N_0$ and $C\subset V$. Then if $C=Y_0$ we have $\widetilde{\bfC}^C_{nT}\subset Y_n$ for any $\lambda>0.$ 
	Thus, if $Y_n=\emptyset$ then $\widetilde{\bfC}^C_{nT}=\emptyset$
	and hence also  $\bfC^C_{nT}=\emptyset$ for any $\lambda>0.$ 
\end{lemma}
\begin{proof}
	In this proof we need the notion of a connecting path which is using the $\bfB'$ background defined in 
	\eqref{connectingpath} via the $w'_{n}(e).$
	If $x\in \widetilde{\bfC}^{C}_{nT}$ then there exists 
	a sequence of times $0 = t_0 < t_1 < \dots < t_{n'} <t_{n'+1}= nT$ 
	and space points $x_0,x_1',\dots, x_{n'}' = x$ with $x_0\in C$ such that
	$({\{x_{k-1}',x_{k}'\}},t_k)\in \Delta^{\inf}$ and $w'_{n_k}(\{x_{k-1}',x_{k}'\})=0$, respectively 
	$X_{\{x_{k-1}',x_{k}'\},n_k}=1$, where $n_k:=\lfloor t_k/T\rfloor$, for all $k \in\{ 1, \dots ,n'\}$ and  $\Delta^{\text{rec}}\cap\big(\{x_k'\}\times[t_k , t_{k+1} )\big)=\emptyset$ for all $k \in \{ 0, \dots , n'\}$. Let for $m \in \{1,\dots n\}$ the position of the path at time $mT$ be denoted by $x_m$, i.e. $x_m= x_{m'}'$ if $mT \in [t_{m'} , t_{m'+1} )$	so that $x_m\in \widetilde{\bfC}_{mT}$ for $m\in \{0,\dots,n\}$. 
	Now if we can show that $x_{m-1}\in \widetilde{\bfC}_{(m-1)T}$ and $x_m\in \widetilde{\bfC}_{mT}$ imply that $x_m\in Y_{m}$ the claim follows since $x_0\in Y_0=C$ by assumption.
	
	So if $x_{m-1}\neq x_m$ it means that the infection must have spread from $x_{m-1}$ to $x_m$ in the time interval $[(m-1)T,mT)$. But we already assumed the existence of a connecting path. Thus, we can  find $0 \leq m' \leq n'$ and $l\geq 1$ such that $x_{m-1}= x_{m'}'$ and $x_{m}= x_{m'+l}'$  	such that 
	$\big(\{x'_{m'+k},x'_{m'+k+1}\},t_{m'+k+1}\big)\in \Delta^{\text{inf}}$
	and $X_{\{x'_{m'+k},x'_{m'+k+1}\},m-1}=1$ for all $k\in \{0,\dots,l-1\}$, and thus by the second rule $x_m\in Y_m$.
	
	If $x_{m-1}=x_m$ then either there was no recovery event in the whole time interval $\big[(m-1)T,mT\big)$, and so by the first rule $x_m\in Y_m$ or the infection must have spread to another vertex and the vertex $x_m$ got reinfected. Then there must have been a vertex $x'_{m'}$ and a time $t_{m'+1}\in[(m-1)T,mT)$ such that $x'_{m'+1}=x_m$ and  $\big(\{x'_{m'},x'_{m'+1}\},t_{m'+1}\big)\in \Delta^{\text{inf}}$ as well as 
	$X_{\{x'_{m'},x'_{m'+1}\},m-1}=1$, 
	and therefore $x_m\in Y_m$ by the second rule.
	
	The second claim follows by the fact that $\bfC^C_{t}\subset\widetilde{\bfC}_t^C$ for all $t\geq 0$. Thus, if $\widetilde{\bfC}^C_{nT}=\emptyset$ then this implies that $\bfC^C_{nT}=\emptyset$.
\end{proof}
Obviously $(U_{x,n})_{(x,n)\in V\times \N_0}$ from \eqref{Udef} is a family of i.i.d.\ random variables with $\Pw(U_{x,n}=1)=e^{-rT}$ and independent of the family $(X_{e,n})_{(e,n)\in \cE\times \N_0}$, which 
consists of independent Bernoulli random variables such that $\Pw(X_{e,n}=1)= 1-\delta_{e}$. The next result 
states a sufficient condition for the extinction of $Y$, which can be proven in exactly the  same way as \cite[Lemma 3.7]{linker2019contact}.
\begin{lemma}\label{Immunization}
	Let $x\in V$. If $\E[|Y_1||Y_0=\{x\}]<1$ then $Y$ goes extinct almost surely for any finite $A\subset V$ as initial state.
\end{lemma}
Now we can finally show Theorem \ref{ImmunThm}.
\begin{proof}[Proof of Theorem~\ref{ImmunThm}]
	We fix $\gamma>0$ and let $\lambda>0$ be arbitrary.
	Since Lemma~\ref{ExtinctionLemma} states that extinction of $Y$ implies extinction of the CPDLP, Theorem~\ref{ImmunThm} follows from Lemma~\ref{Immunization} if we can prove the condition stated in that lemma.
	Fix an arbitrary $x\in V$ and set $Y_0=\{x\}$ as initial value. We can calculate that
	\begin{align}\label{SplitUpOfExpect}
		&\E[|Y_1|]=\E[\1_{\{\exists y\in V: X_{\{x,y\},0}=1\}}|Y_1|]+\Pw\Big(\bigcap_{y\in V}\{X_{\{x,y\},0}=0\}\Big)\E[U_{x,0}].
	\end{align}
	Let us choose $0<\varepsilon<1$ arbitrarily but fixed. For the last term, we find a $T_1>1$ large enough such that 
	\begin{equation}\label{BoundOnDeath}
		\E[U_{x,0}]=e^{-T}<\frac{\varepsilon}{3}
	\end{equation}
	for all $T>T_1$. For the first term	we see that $Y_1$ is actually the connected component containing $x$ formed by a long range percolation model with probabilities $(1-\delta_{e})_{e\in \cE}$ with $\delta_{e}$ as in Proposition~\ref{EdgeBound},	which implies that
	\begin{align}\label{TechnicalDelta}
		\begin{aligned}
			1-\delta_e= 1-e^{-\hat{p}_{e}\hat{v}_{e}T}+\hat{p}_{e}e^{-\hat{p}_{e}\hat{v}_{e}T}+(1-\hat{p}_{e})\hat{p}_{e}\frac{1-e^{-\hat{v}_{e}T}}{e^{\hat{p}_{e}\hat{v}_{e}T}-1}\leq \hat{p}_e\hat{v}_eT+\hat{p}_e+\frac{1}{\hat{v}_eT},
		\end{aligned}	
	\end{align}
	for all $e\in \cE$. Here, we have used that $1-x\leq e^{-x}$ and $1+x\leq e^x$ for $x\geq0$. Recall that $\hat{p}_k=q p_k$. For the remainder of this proof we choose $q=q(T):=T^{-2}$ and see that
	\begin{align}\label{AuxiliaryRates}
		1-\delta_{e}(q(T),T)\leq \frac{1}{T}p_{e} \hat{v}_{e}+\frac{1}{T^2} p_{e}+\frac{1}{\hat{v}_{e}T}=:b_e(T)
	\end{align}
	for all $e\in \cE$. We attach $T$ as an index to $Y_1(T)$ since by the choice of $q$ the probabilities $(1-\delta_{e})_{e\in \cE}$ determining the connected components only depend on the choice of $T$. Next we will show that there exists $T_2>0$ and an $M=M(\varepsilon,T_2)>0$ such that
	\begin{align}\label{BoundBigValues}
		\E[\1_{\{|Y_1(T)|>M\}}|Y_1(T)|]<\frac{\varepsilon}{3}
	\end{align}
	for all $T>T_2$. For this, let $Z(T)$ be the connected component containing $x$ formed by a long range percolation model with probabilities $(b_e(T))_{e\in \cE}$ such that $Y_1(T)\subset Z(T)$ for every $T>0$. This coupling  is possible since \eqref{AuxiliaryRates} holds for all $e\in \cE$.
	By Assumption~\eqref{ExistenceAssumption} and Remark~\ref{Rem:pIsSummable} it follows that $(b_{\{x,y\}}(T))_{y\in V}$ is summable for all $x\in V$ and $T>0$. Furthermore, $b_e(T)$ is decreasing in $T$ and $b_e(T)\to 0$ as $T\to \infty$ for all $e\in \cE$. Therefore, by Lebesgue's  dominated convergence theorem we see that there exists a $T_2\geq T_1$ large enough such that
	\begin{equation*}
		\sum_{y\in V}b_{\{x,y\}}(T)<1
	\end{equation*}
	for all $T\geq T_2$. For this choice of $T_2$ the integrability of $|Z(T_2)|$ follows by Proposition~\ref{lowerBound&integrability}, i.e.~$\E[|Z(T_2)|]<\infty$. Thus, for every $\varepsilon>0$ there exist an $M=M(\varepsilon,T_2)>0$ such that
	\begin{equation*}
		\E[\1_{\{|Z(T_2)|>M\}}|Z(T_2)|]<\frac{\varepsilon}{3}.
	\end{equation*}
	Since $b_e(T)$ is monotone decreasing in $T$ for all $e\in \cE$
	\begin{align*}
		\E[\1_{\{|Z(T)|>M\}}|Z(T)|]\leq\E[\1_{\{|Z(T_2)|>M\}}|Z(T_2)|]<\frac{\varepsilon}{3}
	\end{align*} 
	for all $T>T_2$. Furthermore, since by definition $Y_1(T)\subset Z(T)$ for all $T$ we see that
	\begin{align*}
		\E[\1_{\{|Y_1(T)|>M\}}|Y_1(T)|]<\frac{\varepsilon}{3}.
	\end{align*}
	for all $T>T_2$.  Using this and the bounds \eqref{BoundOnDeath} and \eqref{BoundBigValues} in \eqref{SplitUpOfExpect} we obtain
	\begin{align}
		\nonumber
		\E[|Y_1|]
		<&\E[\1_{\{\exists y\in V: X_{\{x,y\},0}=1\}}|Y_1|]+\frac{\varepsilon}{3}\\
		\nonumber
		\leq& \E[\1_{\{|Y_1|> M\}}|Y_1|]+\E[\1_{\{|Y_1|\leq M\}}\1_{\{\exists y\in V: X_{\{x,y\},0}=1\}}|Y_1|]+\frac{\varepsilon}{3}\\
		<&\frac{\varepsilon}{3}+M\underbrace{\Pw(\{\exists y\in V: X_{\{x,y\},0}=1\}\cap \{|Y_1|\leq M\})}_{\leq \Pw( \exists y\in V: X_{\{x,y\},0}=1)}+\frac{\varepsilon}{3}
		\label{AlmostFinished}
	\end{align} 
	for all $T>T_2$. By using subadditivity of the measure $\Pw$ we get that
	\begin{align*}
		\Pw( \exists y\in V: X_{\{x,y\},0}=1)=\Pw\Big( \bigcup_{y\in V}\{ X_{\{x,y\},0}=1\}\Big)\leq \sum_{y\in V}\big(1-\delta_{\{x,y\}}(q(T),T)\big) \rightarrow 0
	\end{align*}
	as $T \rightarrow \infty$	by Lebesgue's  dominated convergence theorem since
	$1-\delta_{\{x,y\}}(q(T),T)\leq b_{\{x,y\}}(T_2)$ for all $T\geq T_2$ and $(b_{\{x,y\}}(T_2))_{y\in V}$ is summable for every $x\in V$. 	This implies that there exists a $T_3\geq T_2$ such that
	\begin{align}\label{LastIngriedient}
		\Pw( \exists y\in V: X_{\{x,y\},0}=1)\leq \sum_{y\in V}(1-\delta_{\{x,y\}}(q(T),T))<\frac{\varepsilon}{3M}
	\end{align}
	for all $T>T_3$. Now \eqref{AlmostFinished} and \eqref{LastIngriedient} imply that for all $T>T_3$ and thus for all  $q=q(T)=T^{-2}<q_0'$ if we set $q_0':=q(T_3)=T_3^{-2}>0$ we have
	\begin{align*}
		\E[|Y_1|]<\varepsilon<1.
	\end{align*}
	We now let $q_0\geq q_0'>0$ be maximal such that $\lambda_c(\gamma,q)=\infty$ for all $q<q_0$. By the fact that $q \mapsto \lambda_c(\gamma,q)$ is monotone non-increasing we then know that $\lambda_c(\gamma,q)<\infty$ for all  $q>q_0$. To complete the proof it now only remains to show that $\gamma \mapsto q_0(\gamma)$ is a monotone non-increasing function on $(0,\infty)$, i.e.~that for $0<\gamma_1<\gamma_0$ we have $q_0(\gamma_1) \geq q_0(\gamma_0)$. For this
	suppose that we have  $q_0(\gamma_1) < q_0(\gamma_0)$. Then   for any $q \in (q(\gamma_1),q(\gamma_0))$
	we have that $\lambda_c(\gamma_0,q)=\infty$ while $\lambda_c(\gamma_1,q)<\infty$. 
	But this contradicts the fact that by Proposition~\ref{WeakSpeedMonotonicity} we have 
	$\gamma_1^{-1} \lambda_c(\gamma_1,q)> \gamma_0^{-1} \lambda_c(\gamma_0,q)$, and so we are done.
\end{proof}

\subsection{Extinction for slow background speed}\label{SlowSpeedRegime}
In this subsection we study the behavior of the survival probability as $\gamma\to 0$ and show 
Theorem~\ref{AsymptoticSlowSpeedThm}. On general graphs $G=(V,E)$ 
we already obtained partial results on the behavior of the critical infection rate for slow speed of the background process, which we stated in Corollary~\ref{PartialResultsSlowSpeed}: 
There exists a $q_1\in(0,1]$ so that for every $q< q_1$ there exists a $\gamma_0=\gamma_0(q)>0$ such that $\lambda_c(\gamma,q)=\infty$ for all $\gamma<\gamma_0$. Now we restrict ourselves to the one dimensional integer lattice $G=(V,E)$ with $V=\Z$ and $E=\{\{x,y\}\subset \Z: |x-y|=1\}$. In this case we can fully characterize the behavior of the critical infection rate as $\gamma\to 0$ if we assume that \eqref{StrongerAssumption} is satisfied, i.e. 
\begin{equation*}
	\sum_{y\in \N}yv_{\{0,y\}}p_{\{0,y\}}<\infty \quad \text{ and } \quad
	\sum_{y\in \N}yv_{\{0,y\}}^{-1}<\infty. 
\end{equation*}
Obviously this assumption already implies \eqref{ExistenceAssumption}. Furthermore by \eqref{TechnicalDelta} it  follows that
\begin{align}\label{TechnicalDelta2}
	\sum_{y\in \N} y(1-\delta_{\{x,y\}})\leq \sum_{y\in \N} y\Big(\hat{p}_{\{x,y\}} \hat{v}_{\{x,y\}}T+\hat{p}_{\{x,y\}}+\frac{1}{\hat{v}_{\{x,y\}}T}\Big)<\infty
\end{align}
for all $x\in \Z$ due to  \eqref{StrongerAssumption}. In the remainder of this section we focus on proving Theorem~\ref{AsymptoticSlowSpeedThm}. We achieve this by modifying and adapting the strategy used in \cite{linker2019contact}. 

Recall  $\widetilde{\bfC}$  with background $\bfB'$ defined in \eqref{connectingpath}, which is characterized by the  $w'_n(e)$ of Corollary \ref{ComparisonLongRangePerc}. As in the previous section we construct a type of oriented long range percolation model which will be coupled to $\widetilde{\bfC}$ in such a way that if this model goes extinct so does $\widetilde{\bfC}$. Since we know that extinction of $\widetilde{\bfC}$ implies extinction of the CPDLP $\bfC$ this will lead to the proof of Theorem~\ref{AsymptoticSlowSpeedThm}.

One key point of the arguments used in \cite{linker2019contact} was that in an independent percolation model on $\Z$ with $p<1$ no infinite connected component occurs, and thus the percolation almost surely partitions $\Z$ into finite connected components. As we saw in Proposition~\ref{CutPointTheo} the long range percolation exhibits a similar behavior if $\sum_{y\in \N} y(1-\delta_{\{0,y\}})<\infty$.

Recall from Definition~\ref{CutPoint} that a cut-point $m$ for a long range percolation model is a point such that no edge $\{x,y\}$ with $x\leq m<y$ is present. In comparison to the nearest neighbor case one major problem is that in the long range percolation model that we will use, which is defined via the  $w'_n(e)$ of Corollary \ref{ComparisonLongRangePerc}, the presence of cut points at two different vertices is not independent. In fact the events $(\{k \text{ is a cut-point}\})_{k\in\Z}$ are decreasing events, and thus positively correlated by the
FKG inequality, see \cite[Theorem~2.4]{grimmett1999percolation}. But this implies that also the events $(\{k \text{ is no cut-point}\})_{k\in\Z}$ are positively correlated. Therefore, we need to adjust the construction in such a way that we can deal with these correlations.
\begin{definition}
	Let $n,K_0\in \N$ and $T>0$. We call $m\in V$ an \textit{$(n,K_0)$-cut} if $w'_n(\{x,y\})=1$ for all $x\leq m <y$ with $|x-y|\leq 2K_0$.
\end{definition}
We call all edges $e=\{x,y\}$ with 
length  $|x-y| \leq 2K_0$ \emph{short edges} and
denote by $\cB_n^{K_0}$ the $\sigma$-algebra containing information of all $w'_n(e)$ of all short edges in time step $n$, i.e.
\begin{equation}\label{EdgeSigmaAlgebra}
	\cB_n^{K_0}:=\sigma\big(\big\{w'_n(\{x,y\}): |x-y|\leq 2K_0\big\}\big).
\end{equation}
Note that for any $m$ the event that $m$ is an $(n,K_0)$-cut is contained in $\cB_n^{K_0}$.

Now let $r_0\in\N$ and define for $k\in \Z$,
\begin{align*}
	M_k&:=[k(2K_0+r_0),(k+1)(2K_0+r_0)-1]\cap \Z\\
	M_k^{\text{left}}&:=[k(2K_0+r_0),k(2K_0+r_0)+K_0-1]\cap \Z\\
	M_k^{\text{mid}}&:=[k(2K_0+r_0)+K_0,k(2K_0+r_0)+K_0+r_0-1]\cap \Z\\	
	M_k^{\text{right}}&:=[k(2K_0+r_0)+K_0+r_0,(k+1)(2K_0+r_0)-1]\cap \Z
\end{align*}
The collection $(M_k)_{k\in \Z}$ forms a disjoint partition of $\Z$. Furthermore, for every $k\in\Z$ the sets $M_k^{\text{left}}, M_k^{\text{mid}}$ and $ M_k^{\text{right}}$ are disjoint and $M_k=
M_k^{\text{left}}\cup M_k^{\text{mid}}\cup M_k^{\text{right}}$. We also want to remark that $|M_k|=2K_0+r_0$, $|M_k^{\text{mid}}|=r_0$ and $|M_k^{\text{left}}|=|M_k^{\text{right}}|=K_0$. See Figure~\ref{fig:CoverCutPoints} for a illustration. 
\begin{figure}[!b]
	\centering 
	\includegraphics[width=110mm]{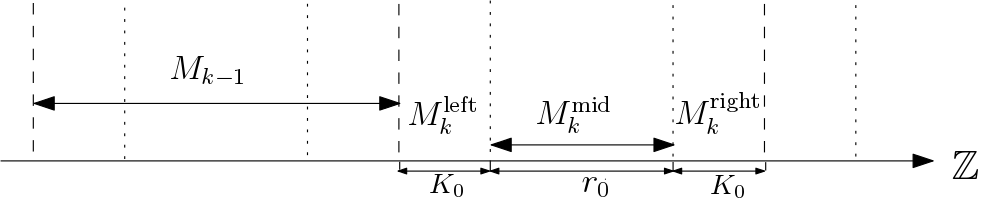}
	\caption{Illustration of the sets $M_{k-1}$, $M_k^{\text{left}}, M_k^{\text{mid}}$ and $ M_k^{\text{right}}$.}
	\label{fig:CoverCutPoints}
\end{figure}
Next we define for $k\in \Z$ and $n \in \N_0$ the random variables
\begin{align}\label{XVariables}
	X_{k,n}:=
	\begin{cases}
		1 &  \text{if no $(n,K_0)$-cut lies in } M_k^{\text{mid}},\\
		0& \text{otherwise.}
	\end{cases}
\end{align}
If $X_{k,n}=0$ then there exists a barrier in $M^{\text{mid}}_k$ during the time interval $[nT, (n+1)T)$ which the infection $\widetilde{\bfC}$ cannot overcome via short edges.

We will now partition the space-time strip $\Z \times [nT,(n+1)T)$ for every $n$, where $T>0$, according to the presence of $(n,K_0)$-cuts. Let $c_{k,n}$ be the rightmost $(n,K_0)$-cut in $M_k^{\text{mid}}\times [nT,(n+1)T)$ and if none is present, then set it equal to the right boundary of $M_k^{\text{mid}}$. Now set $D_{k,n}:=[c_{k-1,n}+1,c_{k,n}]\cap \Z$. We see that $S_{k,n}:=D_{k,n}\times [nT,(n+1)T)$ is a disjoint space-time partition of $\Z\times [0,\infty)$, which depends only on $\cB_n^{K_0}$.
See Figure~\ref{fig:DynamicPartition} for an illustration.
\begin{figure}[t]
	\centering 
	\includegraphics[width=105mm]{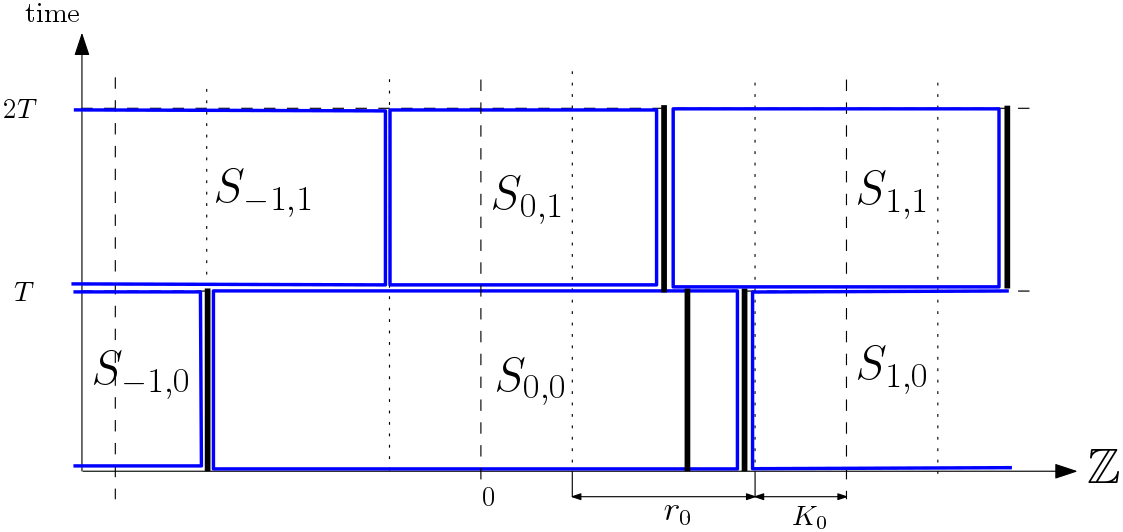}
	\caption{An illustration of a possible partition. The thick black lines represent $(n,K_0)$-cuts and the blue boxes the resulting partition.}
	\label{fig:DynamicPartition}
\end{figure}

The boxes can only be of bounded size and we see from the construction that
\begin{align}\label{MaxMinPar}
	\begin{aligned}
		D_{k,n}\supset D_{k}^{\text{min}}:=&M_{k-1}^{\text{right}}\cup M_k^{\text{left}},\\ 
		D_{k,n}\subset D_{k}^{\text{max}}:=& M^{\text{mid}}_{k-1}\cup M_{k-1}^{\text{right}}\cup M_k^{\text{left}} \cup M^{\text{mid}}_{k}=M^{\text{mid}}_{k-1}\cup D_{k}^{\text{min}} \cup M^{\text{mid}}_{k}.
	\end{aligned}	
\end{align}
This provides us with an upper and lower bound on the number of vertices contained in  $D_{k,n}$, namely $2K_0 \leq |D_{k,n}|\leq 2K_0+2r_0$. We define $S_{k,n}^{\text{min}}:=D_{k}^{\text{min}}\times [nT,(n+1)T)$ and $S_{k,n}^{\text{max}}:=D_{k}^{\text{max}}\times [nT,(n+1)T)$ as the minimal and maximal possible space-time box with $S_{k,n}^{\text{min}}\subset S_{k,n}\subset S_{k,n}^{\text{max}}$.

Recall that $X_{k,n}$ provides us with the information whether it is possible for the infection to traverse $M_{k}^{\text{mid}}$ via \textit{short} edges. So if $X_{k-1,n}=0$ and $X_{k,n}=0$
then the boundaries of $S_{k,n}$ are $(n,K_0)$-cuts and the infection can only leave this box via \textit{long edges} $e=\{x,y\}$ with 
length  $|x-y| \geq  2K_0$. In this case we call the box $S_{k,n}$ \textit{isolated}. In order to describe the possibility of infection in $\widetilde{\bfC}$ via long edges between $S_{k,n}$ and  $S_{l,n}$  we define for $k\neq l \in \Z$ and $n \in \N_0,$
\begin{align}
	\label{Wdef}
	W_{\{k,l\},n}:=
	\begin{cases}
		1 &	\text{if there exists a long edge } e=\{x,y\} \text{ with } |x-y|>2K_0\\
		&	\text{which connects } S_{k,n} \text{ to } S_{l,n} \text{ at some } t\in[nT,(n+1)T),\\
		0& \text{otherwise.} 
	\end{cases}
\end{align}
See Figure~\ref{fig:Connecting} for a visualization in the case $l=k+1.$
\begin{figure}[!b]
	\centering 
	\includegraphics[width=80mm]{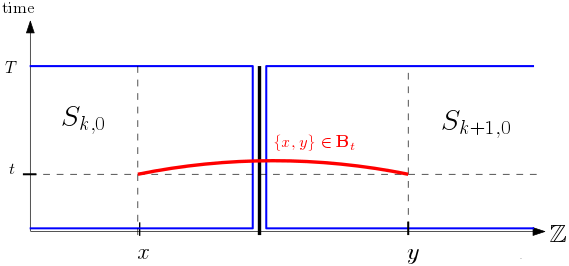}	
	\caption{The thick black line represents again a $(0,K_0)$-cut and the blue boxes a part of the resulting partition. Here we visualize the case when $W_{\{k,k+1\},0}=1$.}
	\label{fig:Connecting}
\end{figure}

Note that by definition $W_{\{k,l\},n}=W_{\{l,k\},n}$, and thus we will assume $k<l$.
The idea is that for large $K_0$ a transmission of the infection via a long edge will be unlikely since they will most likely not be open. In addition, we intend to control the survival via short edges in (isolated) boxes $S_{k,n}$ for $k\in \Z$ and $n \in \N_0$. Recall that a connecting path is a  $\bfB'$-infection path used by $\widetilde{\bfC}$. We define 
\begin{align}
	\label{Udef2}
	U_{k,n}:=
	\begin{cases}
		1 &  \text{if there exists a connecting path starting at } 	nT \text{ that}\\
		&  \text{is ending at } (n+1)T \text{ and is contained in } S_{k,n},\\
		0 & \text{otherwise,} 
	\end{cases}
\end{align}
see Figure~\ref{fig:Surviving} for an illustration. 
If $U_{k,n}=0$ then an infection contained in an isolated box $S_{k,n}$ 
can only survive via transmission along long edges.
\begin{remark}\label{PropertiesOfXWU}
	Let us summarize some properties of the variables we just defined. 
	\begin{enumerate}
		\item The intervals $D_{k,n}$ and boxes $S_{k,n}$ for $k \in \Z$ are measurable with respect to $\cB_n^{K_0}$. For $n \neq n'$ we have that 
		$D_{k,n}$ and $D_{l,n'}$ are independent (and thus also $S_{k,n}$ and $S_{l,n'}$ are independent) for all $k,l \in \Z.$
		\item 
		The $X_{k,n}$ variables from \eqref{XVariables} depend only on short edges of maximal length $2K_0$. Since the minimal distance between $M^{\text{mid}}_{k}$ and $M^{\text{mid}}_{l}$ is larger than $2K_0$ for $k\neq l$ we see that $X_{k,n}$ and $X_{l,n'}$ are independent if $k\neq l$ for all $n,n'\in \N_0$.
		\item
		$U_{k,n}$ from \eqref{Udef2} only depends on edges $\{x,y\}$ with $x,y\in D_{k,n}$ (and recovery events).
		On the other hand, $W_{\{k,l\},n}$ from \eqref{Wdef} only depends on edges $\{x',y'\}$ such that $x'\in D_{k,n}$ and $y'\in D_{l,n}$. 
		Since $D_{k,n}\cap D_{l,n}=\emptyset$ for $k\neq l$ we have that $U_{k',n'}$ and $W_{\{k,l\},n}$ are conditionally independent given $\cB_n^{K_0}$ for all $k',k,l\in \Z$ $(k<l)$ with $n=n'\in \N_0$. If $n\neq n'$ then they are independent.
		\item
		By definition $U_{k,n}$ and $U_{l,n'}$ are conditionally independent given $\cB_n^{K_0}$ for all $k,l\in \Z$  with $n=n'\in \N_0$. If $n\neq n'$ then they are independent.
		\item
		Analogously, the variables $W_{\{l,k\},n}$ and $W_{\{l',k'\},n'}$ are
		conditionally independent given $\cB_n^{K_0}$ for all $\{l,k\}\neq \{l',k'\}$ if $n=n'\in \N_0$. If $n\neq n'$ then they are independent for all choices of $\{l,k\}$ and $\{l',k'\}$.
	\end{enumerate}
	Note that in points 3 to 5 conditioning on $\cB_n^{K_0}$ serves the purpose of knowing what the partition $(S_{k,n})_{k\in \Z}$ in step $n$ look like.
\end{remark}
\begin{figure}[t]
	\centering 
	\includegraphics[width=80mm]{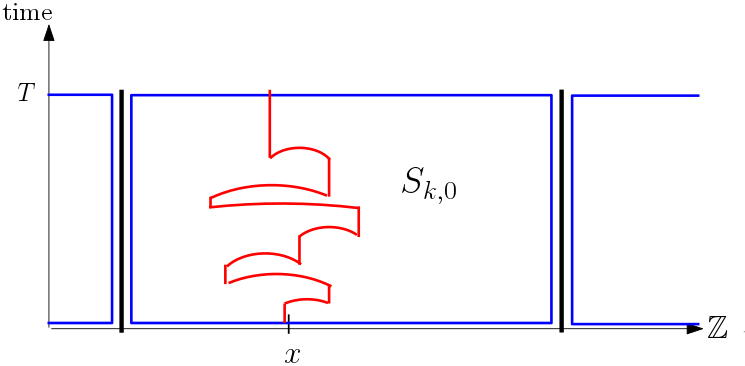}
	\caption{The thick black lines represent again $(0,K_0)$-cuts and the blue boxes are part of the resulting partition. Here we illustrated the case when $U_{k,0}=1$.}
	\label{fig:Surviving}
\end{figure}
We will again define a random graph $G_2$ with vertex set $\Z\times \N_0$ whose  edges are placed according to the following rules which are illustrated in \autoref{fig:ConfiningGraph}:
\begin{enumerate}
	\item If $U_{k,n}=1$ add oriented edges from $(k,n)$ to $(k-1,n+1)$, $(k,n+1)$ and $(k+1,n+1)$.
	\item If $X_{k,n}=1$ add edges as if $U_{k,n}=1$, $U_{k+1,n}=1$ and additionally an unoriented edge between $(k,n)$ and $(k+1,n)$.
	\item If $W_{\{k,l\},n}=1$ add an edge as if $U_{k,n}=1$, $U_{l,n}=1$ and additionally an unoriented edge from $(k,n)$ to $(l,n)$.
\end{enumerate}
If $U_{k,n}=1$ then the infection survives through the space-time box $S_{k,n}$ and it could possibly spread in at least one of the boxes $S_{m,n+1}$ for $m\in\{k-1,k,k+1\}$. If $X_{k,n}=1$ then it could possibly spread to its right neighbor in the time period $[nT,(n+1)T)$ via short edges. If $W_{\{k,l\},n}=1$ for any $l\neq k$ the infection could spread from $S_{k,n}$ to $S_{l,n}$ (or vice versa) via long edges. Note that in the latter two cases we add  oriented edges as in the first case because even if there is no connecting path contained within the respective space-time box 
the infection could still survive (and spread to the nearest neighbours) via open edges between the boxes.
\begin{figure}[t]
	\centering 
	\includegraphics[width=110mm]{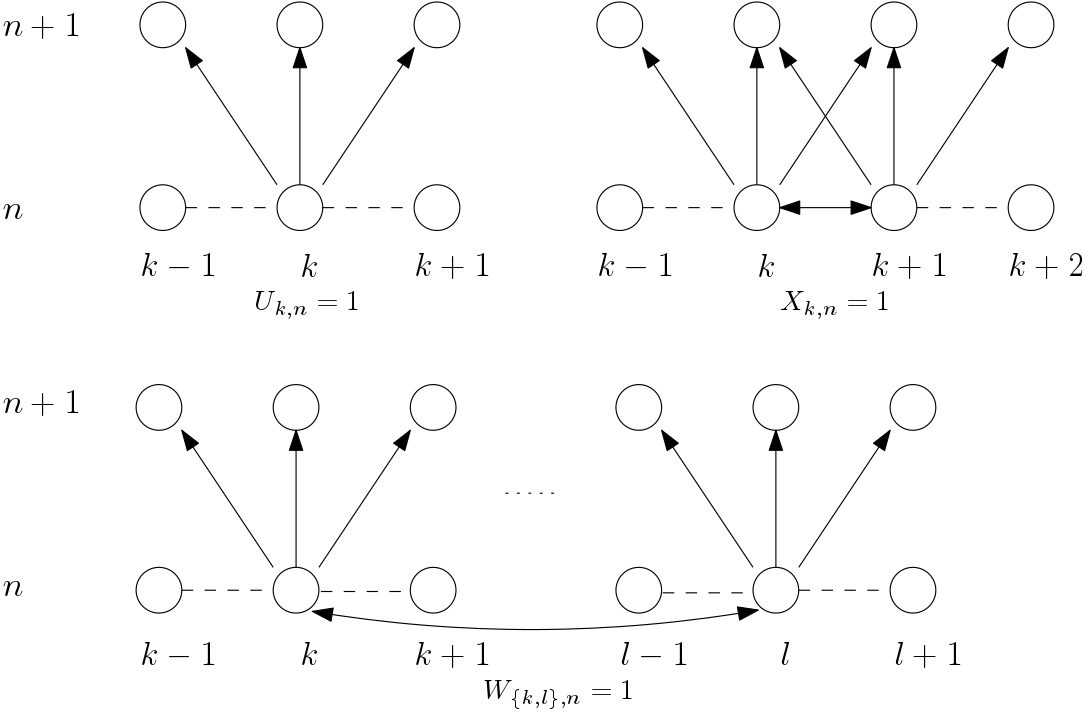}
	\caption{Visualization of the three rules regarding the construction of the graph $G_2$. Solid lines are present edges and dashed lines absent edges.}
	\label{fig:ConfiningGraph}
\end{figure}
Analogous to Definition \ref{validpathY} in the previous section we define a valid path as a path that traverses edges in the direction of their orientation, as well as a corresponding process. 
\begin{definition}
	\label{defn:Zprocess}
	Let $G_2$ be the above constructed random graph. Let $Z_0\subset \Z$ denote the indices of the boxes which contain the initially infected vertices $C \subset \Z$. We say that there exists a \textit{valid path} from $Z_0\times\{0\}$ to a point $(k,n)$  if there exists a sequence $k_0, k_1,\dots k_m=k$ with $k_0\in Z_0$ and $0=n_0\leq n_1\leq \dots\leq n_m=n$ such that there exist an edge in $G_2$ between $(x_k,n_k)$ and $(x_{k+1},n_{k+1})$ for all $k\in \{0,\dots,m-1\}$.
	
	We define the process $Z=(Z_n)_{n\geq 0}$ by letting for all $n\geq 1$ the random set $Z_n=Z_n(U,X,W)$ contain all points $x\in \Z$ for which there exists a valid path from $Z_0\times \{0\}$ to $(x,n)$ in $G_2$. 
\end{definition}
This following lemma and its proof is similar to Lemma~\ref{ExtinctionLemma} in the previous section.
\begin{lemma}\label{ExtinctionOfZImpliesOfC}
	Let $T>0$, $n\in\N_0$ and $C\subset V$. We choose $Z_0$ such that $k\in Z_0$ if $C\cap D_{k,0}\neq \emptyset$. If $x\in \widetilde{\bfC}^{C}_{nT}$ then there exists a $k\in \Z$ so that $x\in S_{k,n}$ and $k\in Z_n$. Thus, if $Z_n=\emptyset$ then $\widetilde{\bfC}^{C}_{nT}=\emptyset$, which in particular implies $\bfC^{C}_{nT}=\emptyset$.
\end{lemma}
\begin{proof}
	Recall the definition of $\widetilde{\bfC}$ and of a connecting path from \eqref{connectingpath}. 
	If $x\in \widetilde{\bfC}^{C}_{nT}$ then for some $x_0\in C$ there must exist a connecting path from $(x_0,0)$ to $(x,nT)$. For $m\in \{0,\dots,n\}$ we denote the position of the connecting path at time $mT$ by $x_m$ so that $x_m\in \widetilde{\bfC}^C_{mT}$ with  $x_n=x$. 
	Since the $(S_{k,n})_{(k,n)\in \Z\times\N_0}$ form a disjoint partition of $\Z\times [0,\infty)$ for every $x_m$ there exists a $k_m=k_m(x_m)$ such that $(x_{m},mT)\in S_{k_m,m}$. To prove the claim it again suffices to show that $x_{m}\in \widetilde{\bfC}^C_{mT}$ and $x_{m+1}\in \widetilde{\bfC}^C_{(m+1)T}$ imply that $k_{m+1}\in Z_{m+1}$ because we have $k_0 \in Z_0$	 by the definition of $Z_0$. 
	\begin{enumerate}
		\item Let us start with the case that $k_m\neq k_{m+1}$. Let $e_1,\dots e_r$ for some $r\in \N$ be the edges present in the connecting path from $(x_{m},mT)$ to $(x_{m+1}, (m+1)T)$ that 
		connect vertices in different space-time boxes.	Let $e_{m'}=\{x',y'\}$ and $t'\in [mT,(m+1)T)$ with $(x',t')\in  S_{k',m}$ and $(y',t')\in  S_{l',m}$ for some $k'<l'$.
		
		If $|x'-y'|>2K_0$ then $W_{\{k',l'\},m}=1$ since 
		the edge $e_{m'}$ must have been open at time $t' \in [mT,(m+1)T)$. Thus, by the third rule $l'\in Z_{m+1}$ if $k'\in Z_m$ (and vice versa).
		
		On the other hand if $|x'-y'|\leq 2K_0$ then $l'=k'+1$ because for any space box $|D_{k,n}|\geq 2K_0$.
		Hence, the boundary between $S_{k',m}$ and $S_{l',m}$ is no $(m,K_0)$-cut since this would prevent an infection to spread via the short edge $e_{m'}$. This implies that $X_{k',m}=1$. 
		Thus, by the second rule $l'\in Z_{m+1}$ if $k'\in Z_m$ (and vice versa).
		
		By applying a combination of the second and third rule to $e_1,\dots, e_r$ it follows that $k_{m+1}\in Z_{m+1}$.
		\item Now we consider the case that $k_m=k_{m+1}$. In this case the infection path is either contained in $S_{k,m}$, and this would imply that $U_{k,m}=1$, or it leaves the box and returns at a later time. By arguing as in point 1. we see that 	either there exists an $l\in \Z$ such that $W_{\{k,l\},m}=1$ or we have $X_{k,m}=1$ or $X_{k-1,m}=1$. 		Thus, $k_{m+1}\in Z_{m+1}$.\qedhere
	\end{enumerate}
\end{proof}
\vspace{-0.5ex}
We again find ourselves in the situation that the process  $Z$ is somewhat easier to handle than the original infection process, but it still hides a lot of dependency structure.
For the remainder of this section we choose $T:=\frac{1}{\gamma}$. 
By the definition of $\delta_{e}$ in Lemma \ref{EdgeBound} this yields\vspace{-1ex}
\begin{equation}\label{IndependenceOfGamma}
	\delta_{e}(\gamma,q,\gamma^{-1})=(1-q p_{e})e^{-q p_{e}v_{e}}\Big(1-q p_{e}\frac{1-e^{-v_{e}}}{1-e^{-q p_{e}v_{e}}}\Big),\vspace{-1ex}
\end{equation}
which is now independent of $\gamma$.

Next we will show that we can choose $r_0,K_0$ and $\gamma$ (or equivalently $T$) in such a way that the probabilities that any of the $X, W$ or $U$ variables are 1  are small. With this we will then show that we can choose $r_0,K_0$ and $\gamma^*$ in such a way that $Z_n$ goes almost surely extinct for all $\gamma<\gamma^*$.

\noindent
\textbf{Bound on the $X$ variables:} Let us recall that
\begin{align}\label{HelpBoundX1}
	\{X_{k,n}=1\}=\{\text{no $(n,K_0)$-cut lies in } M_k^{\text{mid}} \}=\bigcap_{z\in M_k^{\text{mid}}}\bigcup_{\substack{x\leq z < y,\\|x-y|\leq 2K_0}}\{w'_n(\{x,y\})=0\}.
\end{align}
The probability $\Pw(X_{k,n}=1)$ does not depend on $\gamma$ because of 
\eqref{IndependenceOfGamma}. This is important since later, in order to find a bound on $\Pw(U_{k,n}=1)$, we need to vary $\gamma$. 
Since $(w'_n(e))_{e\in \cE}$ is a family of independent Bernoulli random variables we can use 
these variables to define  a long range percolation model with probabilities $b_k:=(1-\delta_{\{0,k\}})$ for all $k\in \Z$.
Therefore, we see with \eqref{HelpBoundX1}  that in the terms of the long range percolation model it holds that
\begin{equation*}
	\{X_{k,n}=1\}\subset \bigcap_{z\in M_k^{\text{mid}}}\bigcup_{x\leq z < y}\{w'_n(\{x,y\})=0\}
	=\{\text{no cut point lies in } M_k^{\text{mid}} \}.
\end{equation*}
We set
\begin{equation}\label{VarBound1}
	\Pw(X_{k,n}=1)\leq \Pw\Big(\bigcap_{z\in M_k^{\text{mid}}}\bigcup_{x\leq z < y}\{w'_n(\{x,y\})=0\}\Big)=:\varepsilon_1(r_0).
\end{equation}
Note that the right hand side only depends on the size $r_0$ of $M_k^{\text{mid}}$ and not its exact location. 
By \eqref{TechnicalDelta2} we know that $\sum_{k\in \Z}^{\infty}kb_k=\sum_{k\in \Z}^{\infty}k(1-\delta_{\{0,k\}})<\infty$. Thus, by Theorem~\ref{CutPointTheo} there exist almost surely infinitely many cut points. But this means that 
\begin{equation}\label{ConvergenceOfVarBound1}
	\varepsilon_1(r_0)\to 0 \quad \text{ as } \quad r_0\to\infty.
\end{equation}
Note that this bound is independent of the choice of $K_0$. This is important since in the next step we derive a bound for the probability $\Pw(W_{\{k,l\},n}=1)$ by choosing $K_0$ accordingly. But the choice of $K_0$ will depend on the choice of $r_0$.

\textbf{Bound on the $W$ variables:} 
For these random variables describing transmission along long edges
we have $W_{\{l,k\},n}=W_{\{k,l\},n}$, which is why we only need to consider $k<l$. We see that
\begin{equation*}
	\{W_{\{k,l\},n}=1\}=\bigcup_{\substack{x\in D_{k,n}, y\in D_{l,n}\\|x-y|>2K_0}}\{w'_n(\{x,y\})=0\}\subset \bigcup_{\substack{x\in D^{\text{max}}_{k}, y\in D^{\text{max}}_{l}\\|x-y|>2K_0}}\{w'_n(\{x,y\})=0\}
\end{equation*}
with the sets $D^{\text{max}}_k$ and $D^{\text{max}}_l $ defined in \eqref{MaxMinPar}.
Note that the right hand side is independent of $\cB_n^{K_0}$ 
defined in \eqref{EdgeSigmaAlgebra}. Thus, for a given $r_0$  we can conclude that
\begin{equation}\label{VarBound2}
	\Pw(W_{\{k,l\},n}=1 |\cB_n^{K_0})\leq \sum_{\substack{x\in D^{\text{max}}_{k},y\in D^{\text{max}}_{l}\\|x-y|>2K_0}} (1-\delta_{\{x,y\}})=:a_{k,l}(K_0,r_0),
\end{equation}
with the right hand side independent of $\gamma$ due to \eqref{IndependenceOfGamma}. By subadditivity and \eqref{VarBound2} we get 
\begin{equation*}
	\Pw(\exists l\neq k: W_{\{k,l\},n}=1 |\cB^{K_0}_n)\leq \sum_{l\neq k}a_{k,l}(K_0,r_0).
\end{equation*}
Next we take a closer look at $D_k^{\text{max}}$ defined in \eqref{MaxMinPar}. We see that $D^{\text{max}}_{k}\cap D^{\text{max}}_{k+1}=M^{\text{mid}}_k$ and if $l>k+1$ then $D^{\text{max}}_{k}\cap D^{\text{max}}_{l}=\emptyset$. Thus, the term $1-\delta_{\{x,y\}}$ appears at most twice in the above sum for any particular edge $\{x,y\}$ with $x \in D^{\text{max}}_{k}$. 
Noting also that $|D_k^{\text{max}}|=2(r_0+K_0)$ and using symmetry 
and translation invariance we see that
\begin{align*}
	\sum_{l\neq k}a_{k,l}(K_0,r_0)=2\sum_{l>k}a_{k,l}(K_0,r_0)&\leq 4|D_k^{\text{max}}|\sum_{y> 2K_0}(1-\delta_{\{0,y\}})\\
	&= 8(K_0+r_0)\sum_{y> 2K_0}(1-\delta_{\{0,y\}}).
\end{align*}
Thus, in summary we obtain for any $k\in \Z$,
\begin{align}\label{VarBound3}
	\Pw(\exists l\neq k: W_{\{k,l\},n}=1 |\cB^{K_0}_n)\leq 8(K_0+r_0)\sum_{y> 2K_0}(1-\delta_{\{0,y\}})=:\varepsilon_2(K_0,r_0).
\end{align}
But since we know from \eqref{TechnicalDelta2} that $\sum_{y\in \N}y(1-\delta_{\{0,y\}})<\infty$. We also obtain that $\sum_{y> 2K_0}y(1-\delta_{\{0,y\}}) \to 0$ as $K_0\to \infty$, and thus it is not difficult to see that 
$\varepsilon_2(K_0,r_0)\to 0$ as $K_0\to \infty$ for every $r_0$. In addition,  if we choose $K_0=r_0$ then we see that also
\begin{equation}\label{K_0=r_0}
	\varepsilon_2(r_0,r_0)\to 0 \quad \text{ as } \quad r_0\to \infty
\end{equation}

\textbf{Bound on the $U$ variables:}
Recall that on every finite graph the classical contact process dies out. We denote by $\tau^{r_0,K_0}_{\text{ext}}$ the extinction time of a classical contact process with infection rate and recovery rate as the CPDLP $(\bfC,\bfB)$ on a complete graph with $2(K_0+r_0)$ vertices, where every vertex is initially infected. Since $|D_{k,n}|\leq 2(K_0+r_0)$ it holds that
\begin{equation}\label{VarBound4}
	\Pw(U_{k,n}=1|\cB^{K_0}_n)\leq \Pw(\tau^{r_0,K_0}_{\text{ext}}>\gamma^{-1})=:\varepsilon_3(K_0,r_0,\gamma).
\end{equation}
For every $\varepsilon>0$ we can choose $\gamma^*=\gamma^*(K_0,r_0)>0$ small enough such that $\Pw(\tau^{r_0,K_0}_{\text{ext}}>\gamma^{-1})<\varepsilon$ for all $\gamma<\gamma^*$, and thus in particular $\varepsilon_3(K_0,r_0,\gamma)\to 0$ as $\gamma\to 0$.

We observe now that $(X_{k,n})_{(k,n)\in\Z\times \N_0}$ are independent random variables and that for
any $n\in\N_0$ the random variables $(X_{k,n})_{k\in \Z}$ are measurable with respect to $\cB^{K_0}_n$. But the family of random variables $(U_{k,n})_{(k,n)\in \Z\times \N_0}$ are  only independent in time (for different $n$), but for a fixed $n$ in the spatial direction (for different $k$)  only independent conditionally on $\cB^{K_0}_n$, see Remark~\ref{PropertiesOfXWU}. The analogous statement holds for 
$\{W_{\{k,l\},n} :k,l \in \Z,k<l,n\in \N_0\}$.
Our next aim is therefore to construct independent upper bounds of the $W$ and $U$ variables, which are also independent of the $X$ variables.
\begin{proposition}\label{BoundsOnXWU}
	Let 	$a_{k,l}(K_0,r_0)$ 
	and $\varepsilon_3(K_0,r_0,\gamma)$ be as in \eqref{VarBound2} 
	and \eqref{VarBound4} and $r_0,K_0>0$ large enough as well as $\gamma>0$ small enough such that $a_{k,l}(K_0,r_0)$,	$\varepsilon_3(K_0,r_0,\gamma)<1$.
	Then there exist independent families
	\begin{align*}
		U':=\{U'_{k,n}:k\in \Z, n\in\N_0\} \quad \text{ and }\quad  W':=\{W'_{\{k,l\},n}:k,l\in \Z,k< l,n\in \N_0\}
	\end{align*}
	of independent Bernoulli random variables with $\Pw(W'_{\{k,l\},n}=1)=a_{k,l}(K_0,r_0)$ 
	and $\Pw(U'_{k,n}=1)=\varepsilon_3(K_0,r_0,\gamma)$ for all $k\neq l$ and all $n\in \N_0$ such that $W_{\{k,l\},n}\leq W'_{\{k,l\},n}$ and $U_{k,n}\leq U'_{k,n}$ almost surely  and such that they are independent of the family $(X_{k,n})_{(k,n)\in\Z\times \N_0}$. 
\end{proposition}
\begin{proof}
	Recall that $\cB_n^{K_0}=\sigma\big(\big\{w'_n(\{x,y\}): |x-y|\leq 2K_0\big\}\big)$ from \eqref{EdgeSigmaAlgebra}. We will now explicitly construct the $U'$ and $W'$ variables. For that we define the random variables
	\begin{equation*}
		p^U_{k,n}:=\Pw(U_{k,n}=0|\cB^{K_0}_n) \quad \text{ and } \quad p^W_{\{k,l\},n}:=\Pw(W_{\{k,l\},n}=0|\cB^{K_0}_n)
	\end{equation*}
	for $l,k\in \Z$ with $k\neq l$ and $n\in \Z_0$. Now let 
	\begin{equation*}
		\{\chi^U_{k,n}: k\in \Z, n\in \N_0\}  \quad \text{ and } \quad  \{\chi^W_{\{k,l\},n}: k,l\in \Z, k\neq l, n\in \N_0 \}
	\end{equation*}
	be two independent families of uniform random variables on $[0,1]$, which are also independent of the $X$, $U$ and $W$ variables. Furthermore, we define the random variables
	\begin{equation*}
		s^{U}_{k,n}:=\frac{1-\varepsilon_3(K_0,r_0,\gamma)}{p^U_{k,n}}\quad \text{ and } \quad s^{W}_{\{k,l\},n}:=\frac{1-a_{k,l}(K_0,r_0)}{p^W_{\{k,l\},n}},
	\end{equation*}
	which are measurable with respect to $\cB_n^{K_0}$.
	Note that by \eqref{VarBound2} and \eqref{VarBound4} it follows that $p^W_{\{k,l\},n}\geq 1-a_{k,l}(K_0,r_0)>0$ and that $p^U_{k,n}\geq 1-\varepsilon_3(K_0,r_0,\gamma)>0$. This yields that $s^{W}_{\{k,l\},n}$ and $s^{U}_{k,n}$ have values in $[0,1]$. Now let $U'=\{U'_{k,n}:k\in\Z,n\in \N_0\}$ be a family of Bernoulli random variables such that $U'_{k,n}=0$ if and only if $U_{k,n}=0$ and $\chi^U_{k,n}\leq s^{U}_{k,n}$. By definition it is clear that $U'_{k,n}\geq U_{k,n}$ and we see that
	\begin{equation}\label{eq:U'Probability}
		\Pw(U'_{k,n}=0|\cB^{K_0}_n)=\Pw(U_{k,n}=0,\chi^U_{k,n}\leq s^{U}_{k,n}|\cB^{K_0}_n)
		= p^{U}_{k,n}s^{U}_{k,n}
		=1-\varepsilon_3(K_0,r_0,\gamma),
	\end{equation}
	where we used in the second equality conditional independence given $\cB_n^{K_0}$, which follows since $s^{U}_{k,n}$ is $\cB_n^{K_0}$ measurable and $\chi^U_{k,n}$ independent of $U_{k,n}$.
	Note that the right hand side is deterministic, and thus it follows that the variables $U'_{k,n}$ are independent of $\cB_n^{K_0}$ for all $k\in \Z$ and $n\in \N_0$.
	
	Analogously, we define the family $W'$ 
	by setting $W'_{\{k,l\},n}=0$ if and only if $W_{\{k,l\},n}=0$ and $\chi^W_{\{k,l\},n}\leq s^{W}_{\{k,l\},n}$ (and otherwise $W'_{\{k,l\},n}=1$). Analogously as in \eqref{eq:U'Probability} it follows that
	\begin{equation}\label{eq:W'Probability}
		\Pw(W'_{\{k,l\},n}=0|\cB^{K_0}_n)= p^W_{\{k,l\},n}s^W_{\{k,l\},n}=1-a_{k,l}(K_0,r_0),
	\end{equation}
	which implies that $W'_{k,n}$ is also independent of $\cB_n^{K_0}$ for all $k,l\in \Z$ with $l\neq k$ and $n\in \N_0$. By taking the expectation in \eqref{eq:U'Probability} and \eqref{eq:W'Probability} we get that
	\begin{equation}\label{eq:Marginals}
		\Pw(U'_{k,n}=0)=1-\varepsilon_3(K_0,r_0,\gamma)\quad \text{ and } \quad \Pw(W'_{\{k,l\},n}=0)=1-a_{k,l}(K_0,r_0),
	\end{equation}
	and therefore $U'_{k,n}$ and $W'_{\{k,l\},n}$ have the correct marginal distribution for all $k,l\in \Z$ with $l\neq k$ and $n\in \N_0$.
	
	It is left to show that these variables are two independent families of independent random variables. We already know by construction that the $U'$ and $W'$ of different time steps $n\neq n'$ are independent. Thus, it suffices to show independence of the variables in the same time step. Therefore, we fix some $n$ and omit the subscript $n$ in the following.
	
	Let $m_1,m_2\in \N_0$.  Let $u_1,\dots u_{m_1}$ and $w_1,\dots w_{m_2}$ be in $\{0,1\}$, and let 
	$k_1,\dots ,k_{m_1}$ be distinct integers as well as  $e_{1},\dots,e_{m_2}$ be distinct edges.
	We need to show that
	\begin{align*}
		&\Pw(U'_{k_1}=u_1,\dots ,U'_{k_{m_1}}=u_{m_1},W'_{e_1}=w_1,\dots, W'_{k_{m_2}}=w_{m_2})\\
		=&\Pw(U'_{k_1}=u_1)\dots \Pw(U'_{k_{m_1}}=u_{m_1})\Pw(W'_{e_1}=w_1)\dots \Pw(W'_{k_{m_2}}=w_{m_2}).
	\end{align*}
	Since we are considering Bernoulli random variables it suffices to consider $u_1=\dots=u_{m_1}=w_1=\dots =w_{m_2}=0$. 
	Now we see that
	\begin{align*}
		&\Pw(U'_{k_1}=0,\dots ,U'_{k_{m_1}}=0,W'_{e_1}=0,\dots ,W'_{k_{m_2}}=0)\\
		=&\E\Big[\Pw\Big(\bigcap_{i=1}^{m_1} \{U_{k_i}=0\}\cap \{\chi^U_{k_i}\leq s^{U}_{k_i}\}\cap \bigcap_{j=1}^{m_2} \{W_{e_j}=0\}\cap \{\chi^W_{e_j}\leq s^{W}_{e_j}\}\Big|\cB^{K_0}\Big)\Big]\\
		=&\E\Big[\Pw\Big(\bigcap_{i=1}^{m_1} \{U_{k_i}=0\}\cap \bigcap_{j=1}^{m_2} \{W_{e_j}=0\}\Big|\cB^{K_0}\Big)\prod_{i=1}^{m_1}\Pw(\chi^U_{k_i}\leq s^{U}_{k_i}|\cB^{K_0})\prod_{j=1}^{m_2}\Pw(\chi^W_{e_j}\leq s^{W}_{e_j}|\cB^{K_0})\Big],
	\end{align*}
	where we again used conditional independence analogously to \eqref{eq:U'Probability} and \eqref{eq:W'Probability}. Thus, we have that
	\begin{align*}
		&\Pw(U'_{k_1}=0,\dots ,U'_{k_{m_1}}=0,W'_{e_1}=0,\dots ,W'_{k_{m_2}}=0)\\
		=&\big(1-\varepsilon_3(K_0,r_0,\gamma)\big)^{m_1}\prod_{j=1}^{m_2}\big(1-a_{e_j}(K_0,r_0)\big)\\
		&\times\E\Big[\Pw\Big(\bigcap_{i=1}^{m_1} \{U_{k_i}=0\}\cap \bigcap_{j=1}^{m_2} \{W_{e_j}=0\}\Big|\cB^{K_0}\Big)  \prod_{i=1}^{m_1}\frac{1}{p^U_{k_i}}\prod_{j=1}^{m_2}\frac{1}{p^{W}_{e_j}}\Big]\\
		=&\big(1-\varepsilon_3(K_0,r_0,\gamma)\big)^{m_1}\prod_{j=1}^{m_2}\big(1-a_{e_j}(K_0,r_0)\big)\\
		=&\Pw(U'_{k_1}=0)\cdots \Pw(U'_{k_{m_1}}=0)\Pw(W'_{e_1}=0)\cdots\Pw(W'_{k_{m_2}}=0),
	\end{align*}
	where we have used in the second to last equality that the $U$ and $W$ variables are conditionally independent given $\cB^{K_0}$ and \eqref{eq:Marginals} in the last equality.
	Note that since the families $U'$ and $W'$  are independent of $\cB^{K_0}$ it follows immediately that they are also independent of the $X$ variables since those are measurable with respect to $\cB^{K_0}$. This concludes the proof
\end{proof}
Now we define a process $(Z'_n)_{n\in \Z}$ 
as in Definition \ref{defn:Zprocess} but as a function of
the random variables $X$, $U'$ and $W'$, which we obtained in Proposition~\ref{BoundsOnXWU}, instead of $X$, $U$ and $W$ used to define $Z$. Due to the monotonicity in the definition of those processes it follows that $Z_n\subset Z'_n$ for all $n\in\N_0$. Thus if $(Z'_n)_{n\in \Z}$ goes extinct almost surely, then the same follows for $(Z_n)_{n\in \Z}$. 
\begin{lemma}\label{ExtinctionSlowSpeed}
	If $\E[|Z'_1||Z'_0=\{0\}]<1$, then $Z'$ dies out almost surely for any finite $A\subset V$ as initial state.
\end{lemma}
\begin{proof}
	Analogously to Lemma~\ref{Immunization}, respectively to \cite[Lemma 3.7]{linker2019contact}.
\end{proof}
Now we are ready to show Theorem~\ref{AsymptoticSlowSpeedThm}, which states that for any $q\in (0,1)$, $C\subset V$ non-empty and finite, and $\lambda>0$ there exists $\gamma^*>0 $ such that $\bfC^{C}$ dies out almost surely for all $\gamma\leq\gamma^*$, i.e.~$\theta(\lambda,\gamma,q,C)=0$ for all $\gamma\leq\gamma^*$. We will also show that this implies that $\lambda_c(\gamma,q)\to \infty$ as $\gamma\to 0$.

\begin{proof}[Proof of Theorem~\ref{AsymptoticSlowSpeedThm}]
	The proof strategy is similar to that of the proof of Theorem~\ref{ImmunThm}, and it again
	suffices to consider $Z'_0=\{0\}$ since the general case follows analogously as in the proof of Theorem~\ref{ImmunThm}. 
	We see that $|Z'_1|\leq 3|Z|$, where $Z$ is the connected component containing the origin $0$ of a long range percolation model (see Section \ref{LongRangePercolationModel}) with probabilities given by  
	\begin{equation*}
		b_{\{k,l\}}=\Pw(W'_{\{k,l\},n}=1) \,\, \text{ and }  \,\,
		b_{\{k,k+1\}}=\Pw(\{W'_{\{k,k+1\},n}=1\}\cup\{X_{k,n}=1\})
	\end{equation*}
	for all  $k,l\in \Z$ with $|k-l| \geq 2$. Note that the constant $3$ comes from the fact that
	any vertex in $Z$, which is connected to the origin at time $0$ via unoriented edges, connects to $3$ vertices at time $1$ via oriented edges,
	see Figure~\ref{fig:ConfiningGraph}. We see that we can again split up the expectation such that
	\begin{align}\label{SplitUp2}
		\begin{aligned}
			\E[|Z|]=&\E\big[\1_{\{X_{1,0}=1\}\cup\{X_{-1,0}=1 \}\cup\{\exists j\in \Z: W'_{\{0,j\},0}=1\}}|Z|\big]\\
			&+\underbrace{\E\big[\1_{\{X_{1,0}=0\}\cap \{X_{-1,0}=0 \}\cap \bigcap_{j\in \Z}\{W'_{\{0,j\},0}=0\}}U'_{0,0}\big]}_{\leq \E[U'_{0,0}]}.
		\end{aligned}
	\end{align}	
	We also know that by \eqref{VarBound1} and \eqref{VarBound2} combined with Proposition \ref{BoundsOnXWU}
	\begin{equation*}
		\Pw(X_{k,n}=1) \leq \varepsilon_1(r_0) \,\, \text{ and }  \,\, \Pw(W'_{\{k,l\},n}=1)\leq a_{k,l}(K_0,r_0)
	\end{equation*}
	so that the $b_{\{k,l\}}$ have bounds that are independent of the choice of $\gamma$, namely, we have for any fixed $k$ that
	\begin{equation*}
		\sum_{l\neq k}b_{\{k,l\}} 
		\leq 2\varepsilon_1(r_0)+ \sum_{l\neq k} a_{k,l}(K_0,r_0).
	\end{equation*}
	From here onwards for the remainder of the proof we choose $K_0=r_0$.
	Now by  \eqref{ConvergenceOfVarBound1} and \eqref{K_0=r_0} it follows that
	\begin{align*}
		\sum_{l\neq k}b_{\{k,l\}}
		\leq 2\varepsilon_1(r_0)+ \sum_{l\neq k} a_{k,l}(r_0,r_0)\to 0
	\end{align*}
	as $r_0\to \infty$. Hence, there exists a constant  $R_1>0$ such that $\sum_{l\neq k}b_{\{k,l\}}
	<1$ for all $r_0\geq R_1$. Thus, by Proposition~\ref{lowerBound&integrability} we know that $|Z|$ is integrable. We add $r_0$ as an index, i.e.~$Z(r_0)$. We can show analogously as in the proof of Theorem~\ref{ImmunThm} that for every $\varepsilon>0$ there exists an $M=M(\varepsilon, R_1)$ such that
	\begin{equation*}
		\E[|Z(r_0)|\1_{\{|Z(r_0)|>M\}}]<\frac{\varepsilon}{3}
	\end{equation*}
	for all $r_0\geq R_1$. Thus, we can conclude that
	\begin{align*}
		&\E[\1_{\{X_{1,0}=1\}\cup\{X_{-1,0}=1 \}\cup\{\exists j\in \Z: W'_{\{0,j\},0}=1\}}|Z(r_0)|]\\
		\leq& \E[\1_{\{|Z(r_0)|>M\}}|Z(r_0)|]+M\big(
		\Pw(X_{1,0}=1)+\Pw(X_{-1,0}=1 )+\sum_{j\in \Z}\Pw(W'_{\{0,j\},0}=1)\big)\\
		\leq& \frac{\varepsilon}{3}+M\Big(2\varepsilon_1(r_0)+ \sum_{l\neq k} a_{k,l}(r_0,r_0)\Big).
	\end{align*}
	Next we again use \eqref{ConvergenceOfVarBound1} and \eqref{K_0=r_0} to see that there must exist a constant $R_2>R_1$ such that $M\big(2\varepsilon_1(r_0)+ \sum_{l\neq k} a_{k,l}(r_0,r_0)\big) <\frac{\varepsilon}{3}$ for all $r_0>R_2$. By \eqref{VarBound4} we can choose $\gamma^*>0$ small enough such that $\E[U_{0,0}]<\frac{\varepsilon}{3}$ for all $\gamma<\gamma^*$. Then it follows with \eqref{SplitUp2} that $\E[|Z|]<3\varepsilon$. Thus, if we choose $\varepsilon<\frac{1}{3}$ we see that
	\begin{align*}
		\E[|Z_1'|]\leq 3\E[|Z|]<1.
	\end{align*}
	By Lemma~\ref{ExtinctionSlowSpeed} it follows that $(Z'_n)_{n\in \N}$ goes extinct almost surely, which implies the same for $(Z_n)_{n\in \N}$ 	since $Z_n\subset Z'_n$ for all $n$ almost surely. Then by Lemma~\ref{ExtinctionOfZImpliesOfC} it follows that $\bfC^{\{x\}}$ goes extinct almost surely and so also $\bfC^{C}$ 	for all finite $C\subset \Z$ and all $\gamma<\gamma^*$.
	In formulas this means that $\theta(\lambda,\gamma,q,C)=0$ for all $\gamma<\gamma^*$.
	
	Finally this implies that $\lim_{\gamma\to 0}\lambda_c(\gamma,q)=\infty$ since otherwise there would exist a $\lambda<\infty$ and a sequence $\gamma_n \rightarrow 0$ such that $\sup_n \lambda_c(\gamma_n, q)<\lambda$. But this would imply that for this $\lambda$ fixed $\theta(\lambda,\gamma_n,q,C)>0$ for all $0<|C|<\infty$ in contradiction to what we just proved. 
\end{proof}
\noindent
\textbf{Acknowledgement.}
We thank Amitai Linker for very helpful discussions at the start of this project, and also Moritz Wemheuer for very useful comments and suggestions. Furthermore, we would like to thank the anonymous referees for carefully reading the manuscript and many suggestions for improvement. MS was partially supported from the LOEWE programme of the state of Hessen (CMMS) in the course of this project.


\end{document}